\documentclass[11pt,reqno]{amsart}
\usepackage{amssymb,amsthm,amsfonts,amstext}
\usepackage{amsmath}
\usepackage{zito}

\usepackage{color}

\makeatletter \@addtoreset{equation}{section}
\@addtoreset{figure}{section}
 \makeatother

\newtheorem{theorem}{Theorem}[section]
\newtheorem{lemma}[theorem]{Lemma}
\newtheorem{proposition}[theorem]{Proposition}

\theoremstyle{remark}
\newtheorem{remark}[theorem]{Remark}

\newcommand{\mc}[1]{{\mathcal #1}}

\newcommand{\bb}[1]{{\mathbb #1}}

\newcommand{\<}{\langle}
\renewcommand{\>}{\rangle}

\newcommand{\eps}{\varepsilon}

\DeclareMathOperator{\divergence}{div}
\DeclareMathOperator{\Osc}{Osc}
\DeclareMathOperator{\TV}{TV}

\newcommand{\tdn}{{\bb T^d_n}}
\newcommand{\omn}{{\Omega_n}}
\newcommand{\etant}{{\eta^n(t)}}
\newcommand{\etanxt}{{\eta_x^n(t)}}
\newcommand{\epso}{{\varepsilon_0}}
\newcommand{\sumxy}{{\!\!\sum_{\substack{x, y \in \tdn\\x \sim y}}\!\!}}
\newcommand{\sumyx}{{\sum_{\substack{y \in \tdn\\y \sim x}}}}
\newcommand{\sumx}{{\sum_{x \in \tdn}}}
\newcommand{\sumy}{{\sum_{y \in \tdn}}}
\newcommand{\sumyz}{{{\!\!\sum_{\substack{y, z \in \tdn\\y \sim z}}\!\!}}}

\newcommand{\muss}{{\mu_{\mathrm{ss}}^n}}
\newcommand{\mussr}[1]{{\mu_{\mathrm{ss}}^{n,#1}}}
\newcommand{\fuss}{{f_{\mathrm{ss}}^n}}
\newcommand{\nurn}{{\nu_\rho^n}}
\newcommand{\rast}{{\rho_\ast}}
\newcommand{\nurnast}{{\nu_{\!\rho_\ast}^n}}
\newcommand{\nurnastr}[1]{{\nu_{\rho_\ast}^{n, #1}}}
\newcommand{\bareta}{{\bar{\eta}}}
\newcommand{\musspr}{{\mu_{\mathrm{ss}}^{n'}}}

\newcommand{\veta}{{\overset{\rightarrow}{\smash{\eta}\vphantom{{}_l}}}}
\newcommand{\vetax}[2]{{\veta^{#1}_{\!#2}}}

\newcommand{\vveta}{{\overset{\leftarrow}{\smash{\eta}\vphantom{{}_l}}}}
\newcommand{\vvetax}[2]{{\vveta^{#1}_{\!#2}}}

\begin{document}
	
\title[CLT for NESS of a reaction-diffusion model]{CLT for NESS of a reaction-diffusion model}

\author{P. Gon\c calves}
\address[P. Gon\c calves]{Center for Mathematical Analysis, Geometry and Dynamical Systems, Instituto Superior T\'{e}cnico, Universidade de Lisboa, 1049-001 Lisboa, Portugal}
\email{pgoncalves@tecnico.ulisboa.pt}

\author{M. Jara}
\address[M. Jara]{Instituto de Matem\'atica Pura e Aplicada, Estrada Dona Castorina 110, 22460-320
	Rio de Janeiro, Brazil}
\email{mjara@impa.br}

\author{R. Marinho}
\address[R. Marinho]{Universidade Federal de Santa Maria, Campus Cachoeira do Sul, Rod. Taufik Germano, 3013, 96503-205, Cachoeira do Sul, Brasil}
\email{rodrigo.marinho@ufsm.br}

\author{O. Menezes}
\address[O. Menezes]{UFBA, Instituto de Matem\'atica, Campus de Ondina, Av. Adhemar de Barros, S/N. CEP 40170-110, Salvador, Brazil
}
\email{ommenezes@gmail.com}

\begin{abstract}
We study the scaling properties of the non-equilibrium stationary states (NESS) of a reaction-diffusion model. Under a suitable smallness condition, we show that the density of particles satisfies a law of large numbers with respect to the NESS, with an explicit rate of convergence, and we also show that at mesoscopic scales the NESS is well approximated by a local equilibrium (product) measure, in the total variation distance. In addition, in dimensions $d \leq 3$ we show a central limit theorem (CLT) for the density of particles under the NESS. The corresponding Gaussian limit can be represented as an independent sum of a white noise and a massive Gaussian free field, and in particular it presents macroscopic correlations.
\end{abstract}


\maketitle

\section{Introduction}
\label{s1}

Non-equilibrium stationary states (NESS) describe the large-time behavior of stochastic interacting systems that are kept out of equilibrium by the action of external forces.
A main characteristic of NESS is the presence of \emph{steady flows}, which can manifest themselves as energy flows, particle flows or mass flows. NESS typically appear when a closed system is kept in contact with several \emph{reservoirs} with different thermodynamic parameters, such as temperature or chemical potential.

When stochastic interacting systems are modelled by Markov chains, the NESS have a simple probabilistic interpretation. An invariant, ergodic measure of a Markov chain is a NESS if it is not reversible with respect to the generator of the Markov chain.

In this article, we propose a general mathematical framework to describe the NESS of \emph{driven diffusive systems}. In order to keep the technical parts at an acceptable level, we consider one of the simplest models of driven diffusive systems, the so-called \emph{reaction-diffusion model} introduced in \cite{D-MFerLeb}, with \emph{quadratic} reaction term. This model has a single order parameter, the density of particles. Our main result is a description at the level of the central limit theorem (CLT) of the fluctuations of the density of particles with respect to the NESS. We show that, in dimension $d \leq 3$ and under a near-equilibrium condition, the scaling limit of these fluctuations is described by a Gaussian process, which is a mixture of a white noise and a massive Gaussian free field, and in particular it presents non-local spatial correlations. These non-local correlations are a signature of NESS, and their precise description is one of the main goals of a mathematical treatment of NESS. 

At the level of the law of large numbers and the large deviations principle, the so-called \emph{macroscopic fluctuation theory} (MFT) provides a fairly complete description of the fluctuations of the density of particles with respect to the NESS, see \cite{BerD-SGabJ-LLan, LanTsu}. However, at the level of the CLT, a description of the fluctuations is only available for a handful of models, \cite{Spo, LanMilOll, FGN, GonJarMenNeu}. Although we prove our main result only for an example of a reaction-diffusion model, we claim that our framework can be used to derive a CLT for NESS of general driven diffusive models in dimensions $d \leq 3$ under a near-equilibrium condition.~The restriction on the  dimension is necessary, but technical, while the near-equilibrium condition is necessary in the following sense. For reaction-diffusion models with cubic interactions, as the one considered in \cite{D-MFerLeb}, a phase transition appears at bifurcation points of the effective reaction function defined in \eqref{chanaral}. Since our methodology also applies for the model considered in \cite{D-MFerLeb}, it is natural to expect a restriction in terms of the parameters of the model. 

The proof of the CLT for NESS is based on an explicit estimate of the \emph{relative entropy} between the NESS and a product Bernoulli measure, valid for any dimension $d$. Our proof uses Yau's relative entropy method \cite{Yau}, recently improved in \cite{JarMen} and that yields quantitative estimates for the convergence results. In dimension $d=1$, the entropy estimate is uniform on the size of the system. In dimensions $d \geq 2$, the estimate is not uniform on the size of the system. As a consequence of our results we are also able to show that the fluctuations of the NESS are absolutely continuous with respect to a white noise for any $d\leq 3$. 

Since our entropy estimates are explicit, they allow  \emph{quantitative} estimates on the law of large numbers for the NESS in every dimension $d \geq 1$. As far as we know, these estimates are the first example in the literature of what we call \emph{quantitative hydrostatics}. Using the translation invariance of the NESS, we also show that in boxes of mesoscopic size, the structure of the NESS is indistinguishable from a Bernoulli product measure, a fact known in the literature as \emph{local equilibrium}. 

\subsection*{Outline} In Section \ref{s2} we present the reaction-diffusion model we consider here and we state our main results. The main objective of Section 3 is to present a complete proof of Theorem \ref{t2}, which is the estimate of the relative entropy of the NESS with respect to the Bernoulli product measure $\nurnast$, where $\rast$ is the unique zero in $[0,1]$  of the function $F$ defined in \eqref{chanaral}. The proof combines Yau's inequality, reviewed in Section \ref{s3.1}, with the log-Sobolev inequality, reviewed in Section \ref{s3.2}, and with  Lemma \ref{main}, introduced in \cite{JarMen} and called \emph{main lemma}. We present here a complete proof of the main lemma for two reasons. First, in our setting we only use translation-invariant reference measures, which makes the proof of Theorem \ref{main} easier to follow in our context. And second, we need to keep track of all constants in the estimates, so that later on we can tune the parameter $\lambda$ to make the constants in Lemma \ref{main} small enough. With the entropy estimate, the proof of Theorems \ref{t1} and \ref{loceq} can be completed.
In Section \ref{s4} we prove Theorem \ref{t3}. In order to prove it, we use the \emph{dynamical approach} introduced in \cite{BerD-SGabJ-LLan} and used in \cite{FarLanMou}
to derive a large deviations principle for the NESS of boundary-driven diffusive systems. In dimension $d=1$, the idea is simple. The entropy estimate of Theorem \ref{t2} implies tightness of the fluctuations of the density of particles around its hydrostatic limit. Moreover, any convergent subsequence satisfies the hypothesis of Theorem 2.4 in \cite{JarMen}. From these two facts, the proof of Theorem \ref{t3} follows as in \cite{GonJarMenNeu}. In dimensions $d=2, 3$, the entropy estimate of Theorem \ref{t2} is not good enough to imply tightness directly. Therefore, we use a more indirect approach, based on Duhamel's formula for solutions of the SPDE \eqref{SPDE}.

\section{Model and statement of results}
\label{s2}
Let $n \in \bb N$ be a scaling parameter and let $\tdn := \bb Z^d / n \bb Z^d$
be the discrete torus of size $n$ and dimension $d$. Let $\omn:= \{0,1\}^\tdn$ be
the state space of a Markov chain $(\etant; t  \geq 0)$ to be described below.
The elements $x \in \tdn$ are called \emph{sites} and the elements $\eta =
(\eta_x; x \in \tdn) \in \omn$ are called \emph{particle configurations}. We say
that the configuration $\eta \in \omn$ has a particle at site $x \in \tdn$ if
$\eta_x =1$. If $\eta_x =0$, we say that the site $x$ is \emph{empty}. 

For $x,y \in \tdn$ and $\eta \in \omn$, let $\eta^{x,y} \in \omn$ be given by
\[
 \eta^{x,y}_z :=
\left\{
\begin{array}{r @{\;;\;}l}
 \eta_y & z =x,\\
\eta_x & z =y,\\
\eta_z & z \neq x,y.
\end{array}
\right.
\]
In other words, $\eta^{x,y}$ is the particle configuration obtained from $\eta$
by exchanging the values of $\eta_x$
and $\eta_y$. For $f: \omn \to \bb R$ and $x,y \in \tdn$, let $\nabla_{\!x,y} f:
\omn \to \bb R$ be given by
\[
 \nabla_{\!x,y} f(\eta) := f(\eta^{x,y}) -f(\eta)
\]
for every $\eta \in \omn$. 

For $x \in \tdn$ and $\eta \in \omn$, let $\eta^x \in \omn$ be given by 
\[
 \eta^{x}_z :=
\left\{
\begin{array}{c @{\;;\;}l}
1-\eta_x & z =x,\\
\eta_z & z \neq x,
\end{array}
\right.
\]
that is, $\eta^x$ is the particle configuration obtained from $\eta$ by changing
the value of $\eta_x$. For $f: \omn \to \bb R$ and $x \in \tdn$, let $\nabla_x
f: \omn \to \bb R$ be given by
\[
\nabla_x f(\eta) := f(\eta^x) -f (\eta)
\]
for every $\eta \in \omn$. 

Let $x,y \in \tdn$. We say that $x \sim y$ if $|x_1-y_1|+\dots+|x_d-y_d|=1$.
Let $a,b > 0$ and $\lambda >-a$. For each $x \in \tdn$, let $c_x = c_x(a,b,\lambda,d): \omn \to
[0,\infty)$ be given by
\begin{equation}
\label{iquique}
c_x(\eta) := \Big(a+ \frac{\lambda}{2d} \sum_{\substack{y \in \tdn \\ y \sim x}}
\eta_y\Big) (1-\eta_x) +b \eta_x 
\end{equation}
for every $\eta \in \omn$. Observe that the conditions on $a,b$ and $\lambda$
imply that $c_x(\eta) \geq \epso >0$ for every $x \in \tdn$ and every $\eta \in
\omn$, where $\epso:= \min\{a, a+\lambda,b\}$. This condition makes the Markov chain defined
below \emph{irreducible}. Later on the parameter $\lambda$ will be chosen smaller than some constant $\lambda_c= \lambda_c(a,b,d)$.
For each $f: \omn \to \bb R$, let $L_n f : \omn \to \bb R$ be given by\footnote{Here and below, $\displaystyle{\sumxy}$ indicates a sum over \emph{unordered} pairs $x,y \in \tdn$.}
\begin{equation}
\label{tocopilla}
L_n f := n^2 \sumxy \nabla_{\!x,y} f + \sum_{x \in \tdn} c_x \nabla_x f.
\end{equation}
The linear operator $L_n$ defined in this way turns out to be the generator of a
Markov chain in $\omn$ that we denote by $(\etant; t \geq 0)$. The sequence
of chains $(\etant; t \geq 0)_{n \in \bb N}$ is an example of what is known in
the literature as a \emph{reaction-diffusion model}.

As pointed out above, under the condition $\epso >0$, the chain $(\etant ; t \geq 0)$ is irreducible,
and therefore it has a unique invariant measure that will be denoted by $\muss$. 
For $\lambda \neq 0$, the measure $\muss$ is not reversible with
respect to $(\etant ; t \geq 0)$. For this reason, we say that the measure $\muss$
is a \emph{non-equilibrium stationary state} (NESS). Our main goal is the
description of the scaling limits of the density of particles with respect to
the measures $(\muss; n \in \bb N)$. For the moment we will not be very specific
about what we understand by the density of particles; a proper definition will
be included where needed.

\subsection*{Notation} We will denote by $\bb P^n$ the law of $(\etant; t \geq
0)$ on the space $\mc D([0,\infty), \omn)$ of \emph{c\`adl\`ag} trajectories,
and we will denote by $\bb E^n$ the expectation with respect to $\bb P^n$.
Whenever we need to specify the initial law $\mu^n$ of the chain $(\etant; t \geq
0)$, we will use the notations $\bb P^n_{\!\mu^n}$, $\bb E^n_{\mu^n}$. 
We will denote by $C$ a finite and positive constant depending only on $a,b,d$, which may change from line to line.

\subsection{Hydrodynamic limit and hydrostatic limit} 

The generator $L_n$ combines an \emph{exclusion dynamics}, corresponding to the
operator $L_n^{\mathrm{ex}}$ given by
\[
L_n^{\mathrm{ex}} f:= n^2 \sumxy \nabla_{\!x,y} f,
\]
with a \emph{reaction dynamics}, corresponding to the operator $L_n^{\mathrm r}$
given by
\[
L_n^{\mathrm r} f := \sum_{x \in \tdn} c_x \nabla_x f.
\]
Observe that the reaction rates defined in \eqref{iquique} correspond to a
\emph{contact process} with an additional creation term avoiding absorption at
density zero. The reaction-diffusion model was introduced in \cite{D-MFerLeb}
with a reaction term corresponding to a stochastic Ising model. The key
observation of \cite{D-MFerLeb} is the following. In order that both parts of the
dynamics, the exclusion part and the reaction part, have a non-trivial effect
on the scaling limits of the density of particles, one must include a factor
$n^2$ in front of the exclusion part of the dynamics, as in the definition
of $L_n$ given in \eqref{tocopilla}. In \cite{D-MFerLeb} the authors derive the
so-called \emph{hydrodynamic limit} of the reaction-diffusion model, which we now describe. For $\rho \in [0,1]$, let $\nurn$ be the Bernoulli product measure
in $\omn$ of density $\rho$:
\[
\nurn(\eta) := \prod_{x \in \tdn} \big(  \eta_x \rho +(1-\eta_x)(1-\rho)\big)
\]
for every $\eta \in \omn$. Fix $x \in \tdn$ and let $F: [0,1] \to \bb R$ be
given by 
\begin{equation}
\label{chanaral}
F(\rho) := \int L_n \eta_x\, d \nurn
\end{equation}
for every $\rho \in [0,1]$. Since the operator $L_n$ and the measures $(\nurn;
\rho \in [0,1])$ are translation invariant, $F$ does not depend on $x$. Observe
that
\[
L_n \eta_x = n^2 \sumyx (\eta_y -\eta_x) + c_x(\eta) (1-2\eta_x)
      = n^2 \sumyx (\eta_y -\eta_x) + \Big(a+ \frac{\lambda}{2d} \sumyx \eta_y
\Big) (1-\eta_x) - b\eta_x.
\]
Therefore, 
\begin{equation}
\label{antofagasta}
F(\rho) = (a+\lambda \rho) (1-\rho) - b \rho 
\end{equation}
for every $\rho \in [0,1]$. We will also need to introduce the function $G: [0,1] \to \bb R$ given by
\begin{equation}
\label{atacama}
G(\rho) := (a+\lambda \rho) (1-\rho) + b \rho 
\end{equation}
for every $\rho \in [0,1]$. Observe that $G(\rho) = F(\rho+ 2b \rho$, although this relation is not relevant in what follows.

Let $\mc C(\bb T^d; \bb R)$ be the set of continuous functions $f: \bb T^d \to \bb R$ and 
let $(\pi_t^n; t \geq 0)$ be the family of measures defined by duality as
\[
\pi_t^n(f) := \frac{1}{n^d} \sumx \etanxt f\big(\tfrac{x}{n}\big)
\]
for every $f \in \mc C(\bb T^d; \bb R)$. We have that

\begin{proposition}[Hydrodynamic limit \cite{D-MFerLeb, KipOllVar}]
 \label{p1}
Let $(\mu^n; n \in \bb N)$ be a sequence of probability measures in $\omn$ and
let $u_0: \bb T^d \to [0,1]$ be a given measurable function. Assume that for every $f \in \mc C(\bb
T^d; \bb R)$,
\[
\lim_{n \to \infty} \frac{1}{n^d} \sumx \eta_x f \big( \tfrac{x}{n}\big) = \int_{\bb T^d}
u_0(x) f(x) dx
\]
in probability with respect to $(\mu^n; n \in \bb N)$. For every $f \in \mc
C(\bb
T^d; \bb R)$ and every $t \geq 0$,
\[
\lim_{n \to \infty} \pi_t^n(f) = \int_{\bb T^d}  u_t(x) f(x) dx
\]
in probability with respect to $\bb P^n_{\mu^n}$, where $(u_t; t \geq 0)$ is the solution of the
hydrodynamic equation
\begin{equation}
\label{ECHID}
\partial_t u = \Delta u + F(u)
\end{equation}
with initial condition $u_0$.
\end{proposition}

This result is known in the literature as the hydrodynamic limit of the
reaction-diffusion model $(\etant; t \geq 0)_{n \in \bb N}$. The hydrodynamic equation \eqref{ECHID} is an example of a \emph{reaction-diffusion equation}. The
intuition behind the definition of the function $F$ appearing in \eqref{ECHID}
is the following. If we look at the dynamics on a box of size $1 \ll \ell \ll
n$, the exclusion dynamics is much faster than the reaction dynamics; in fact,
the exclusion dynamics takes times of order $\mc O(\frac{\ell^2}{n^2})$ in order
to equilibrate the density of particles on a box of size $\ell$, and the  first
jump of the reaction dynamics happens after times of order $\ell^{-d}$.
Therefore, if $\ell \ll n^{\frac{2}{2+d}}$, the exclusion dynamics equilibrates
the density of particles between consecutive jumps of the reaction dynamics. Since the
Bernoulli product measures are invariant under the exclusion dynamics, it is reasonable to assume that 
at least on boxes of size $\ell \ll n^{\frac{2}{2+d}}$, the law of the process is close to product. Therefore it makes sense to define the function $F$ as the average reaction rate
with respect to $\nurn$. This intuition will be made rigorous for the invariant measure $\muss$ in Theorem \ref{loceq} below. Theorem \ref{loceq} also shows that the size $\ell$ of the boxes on which this product approximation is accurate is larger than $n^{\frac{2}{2+d}}$. 

Observe that $F(0) =a>0$, $F(1) =-b<0$ and $F$ is a polynomial of
degree $2$. Therefore, $F$ has a unique zero $\rast$ in $[0,1]$, which satisfies
$0< \rast < 1$ and $F'(\rast) <0$, from where we conclude that $\rast$ is also
stable. The density $\rast$ can be explicitly computed, but its exact formula plays no role in our proofs; we will
only use the property
\[
\lim_{\lambda \to 0} \rast = \frac{a}{a+b} \in (0,1).
\]
This limit can be directly verified from the explicit formula for $\rast$ or it can be deduced from
the fact that $F(\rho) \to a -\rho(a+b)$ as $\lambda \to 0$.  

In view of the stationarity properties of the hydrodynamic equation
\eqref{ECHID}, it is reasonable to postulate that, with respect to $\muss$, the
density of particles is close to $\rast$. However, a proof of this claim
requires a non-trivial exchange of limits. Such exchange of limits has been
justified in \cite{LanTsu}, elaborating over an idea introduced in
\cite{MouOrl, FarLanMou}:

\begin{proposition}[Hydrostatic limit]
 \label{p2}
For every $f \in \mc C(\bb T^d; \bb R)$, 
\[
\lim_{n \to \infty} \frac{1}{n^d} \sumx \eta_x f \big( \tfrac{x}{n} \big) = \int_{\bb T^d}
\rast f(x) dx
\]
in probability with respect to $\{\muss; n \in \bb N\}$.
\end{proposition}

Our first result is a quantitative version of Proposition \ref{p2}. For $n,d \in
\bb N$, let $g_d(n)$ be defined as
\begin{equation}
\label{gd}
g_d(n):=
\left\{ 
\begin{array}{c@{\;;\;}r}
 n & d=1,\\
 \log n & d=2, \\
 1 & d \geq 3.
\end{array}
\right. 
\end{equation}
Observe that $g_d(n)$ corresponds to the order of magnitude of the Green's
function at $0$ of a simple symmetric random walk absorbed at the boundary of a
box of size $n$, centered at the origin.

\begin{theorem}
\label{t1}
 There exist $\lambda_c = \lambda_c(a,b,d)$ positive and $C=C(a,b,d)$ finite such
that for every $f \in \mc C(\bb T^d; \bb R)$, every $n \in \bb N$ and every
$\lambda \in [-\lambda_c,\lambda_c]$,
\[
\int \Big(\frac{1}{n^d} \sumx (\eta_x - \rast) f \big( \tfrac{x}{n} \big)
\Big)^2 \muss(d \eta) 
    \leq \frac{C g_d(n)}{n^2} \cdot \frac{1}{n^d} \sumx f
\big(\tfrac{x}{n}\big)^2.\]
\end{theorem}

For every probability measure $\mu$ in $\Omega_n$ and every density $f$ with respect to $\mu$, let $H(f; \mu)$ denote the relative entropy of $f$ with respect to $\mu$:
\[
H(f; \mu) := \int f \log f d \mu.
\]
Theorem \ref{t1} is a consequence of the following estimate:

\begin{theorem}
\label{t2}
 Let $\fuss$ be the density of $\muss$ with respect to $\nurnast$. There exist  $\lambda_c = \lambda_c(a,b,d)$ positive and $C = C(a,b,d)$ finite such that
\[
H( \fuss ; \nurnast) \leq C  n^{d-2} g_d(n)
\]
for every $n \in \bb N$ and every $\lambda \in [-\lambda_c,\lambda_c]$.
\end{theorem}

\begin{remark}
 The value of $\lambda_c$ is the same in Theorems \ref{t1} and \ref{t2}, but the
value of $C$ can be different.
\end{remark}

Given two probability measures $\mu$, $\nu$ in $\Omega_n$, the relative entropy of $\mu$ with respect to $\nu$ is defined as $H( \frac{d \mu}{d \nu} ; \nu)$ if $\mu$ is absolutely continuous with respect to $\nu$ and $+\infty$ otherwise. Although relative entropy is not a distance (it is not even symmetric), it is widely used in the literature as a measure of closeness between probability measures. Another popular way to measure the closeness of two probability measures is through the \emph{total variation distance}.
The total variation distance between $\mu$ and $\nu$ is defined as
\[
d_{\TV}(\mu,\nu):= \frac{1}{2} \sum_{\eta \in \Omega_n} |\mu(\eta)-\nu(\eta)|.
\]
Total variation and relative entropy are related by \emph{Pinsker's inequality}:
\[
2 d_{\TV}(\mu ,\nu)^2 \leq  H \Big(\frac{d\mu}{d \nu} ; \nu \Big).
\]
Observe that the relative entropy bound of Theorem \ref{t2} is not strong enough to conclude that the measures $\muss$ and $\nurnast$ are close. This is actually expected, since we will see in Theorem \ref{t4} below that the scaling limits of the density of particles are different under $\muss$ and $\nurnast$. 
Nevertheless, we will prove that the bound of Theorem \ref{t2} is good enough to show a strong version of what is known in the literature as \emph{conservation of local equilibrium}, namely, that restricted to boxes of mesoscopic size, the measures $\muss$ and $\nurnast$ are close in total variation. In order to be precise, we need a few definitions.
Let $\mc P: \bb Z^d \to \tdn$ be the universal cover of $\tdn$.
For $R \in \bb N$, let $B_R := \{x \in \bb Z^ d; |x_i| \leq R, i=1,\dots,d\}$. For $n > 2R+1$, let $\Pi_R: \Omega_n \to \{0,1\}^{B_R}$ be the canonical projection: $(\Pi_R \eta)_x = \eta_{\mc P(x)}$ for every $x \in B_R$. Let $\mussr{R}$ be the push-forward of $\muss$ under $\Pi_R$ and let $\nurnastr{R}$ be the push-forward of $\nurnast$ under $\Pi_R$. We will prove the following result:

\begin{theorem}
\label{loceq}
If
\[
\frac{R_n g_d(n)^{1/d}}{n^{2/d}} \to 0 \text{ as }  n \to  \infty
\]
then
\[
d_{\TV} ( \mussr{R_n} , \nurnastr{R_n} ) \to 0 \text{ as } n \to \infty.
\]
\end{theorem}

This theorem states that on a window of mesoscopic size $R_n$, the NESS $\muss$ is asymptotically indistinguishable from a Bernoulli product measure of density $\rast$.

\subsection{Fluctuations and CLT for NESS} Observe that Theorem \ref{t1} can be
understood as a law of large numbers for the density of particles with respect
to $(\muss; n \in \bb N)$. Therefore, it is natural to study the fluctuations
 of the density of particles around its hydrodynamic limit. The \emph{density
fluctuation field} is the function $X^n : \omn \to \mc S'(\bb T^d)$ defined by
duality as
\[
X^n(\eta, f) : = \frac{1}{n^{d/2}} \sumx (\eta_x - \rast) f\big( \tfrac{x}{n} \big)
\]
for every $f \in \mc C^\infty(\bb T^d; \bb R)$ and every $\eta \in \omn$.
Although $X^n$ is well defined as a signed measure in $\bb T^d$, it is more
convenient to think about $X^n$ as a \emph{random distribution}. In a
probabilistic context, on which $\mc L^2$-norms are related to variances,
\emph{Sobolev spaces} are a specially convenient choice of topology for the
image of $X^n$. It will be convenient to define Sobolev spaces in terms of Fourier transforms. 
For $f \in \mc L^1(\bb T^d)$, let $\hat{f}: \bb Z^d \to \bb C$
be given by
\begin{equation}
\label{calama}
\hat{f} (k) := \int_{\bb T^d} e^{-2 \pi i k x} f(x) dx
\end{equation}
for every $k \in \bb Z^d$, that is, $\hat f$ is the \emph{Fourier transform} of
$f$. For $k=(k_1,\dots k_d) \in \bb Z^d$, let us write $\|k\| := (k_1^2+\dots k_d^2)^{1/2}$. For $f \in \mc C^\infty(\bb T^d; \bb R)$ and $m \in \bb N$, let
\[
\|f\|_{\mc H^m} := \Big( \sum_{k \in \bb Z^d} |\hat f (k)|^2 (1+
\|k\|^2)^m\Big)^{1/2}.
\]
Observe that $\|f\|_{\mc H^m}<+\infty$ for every $m \in \bb N$ and every $f \in
\mc C^\infty(\bb T^d; \bb R)$, and observe that the space $(C^\infty(\bb T^d; \bb
R),\|\cdot\|_{\mc H^m})$  is pre-Hilbert. The so-called \emph{Sobolev space} $\mc H^m=\mc
H^m(\bb T^d)$ of order $m$ is defined as the closure of $\mc C^\infty(\bb T^d;
\bb R)$ with respect to $\|\cdot\|_{\mc H^m}$. All the spaces $\mc H^m$ are
Hilbert spaces, and by Parseval's identity, $\mc H^0 = \mc L^2(\bb T^d)$.
Moreover, for every $m \neq 0$, the spaces $\mc H^m, \mc L^2(\bb T^d)$ and $\mc
H^{-m}$ form a \emph{Gelfand triple}. Observe as well that $\mc H^{m'}$ is
compactly contained in  $\mc H^m$ if $m < m'$. Since the Dirac $\delta$ distribution belongs to $\mc H^{-m}$ for every $m > d/2$, we see that
$X^n$ is a random variable in $\mc H^{-m}$ for every $m > d/2$.

Let $\chi(\rast) := \rast(1-\rast)$ be the \emph{mobility} of the reaction-diffusion model. This quantity is one of the thermodynamic variables appearing in MFT  \cite{BerD-SGabJ-LLan2}.
Now we can state our main result:

\begin{theorem}
 \label{t3}
Let $d \leq 3$ and let $\eta$ have law $\muss$. There exists $\lambda_c=\lambda_c(a,b,d) >0$ such that for every
$\lambda \in [-\lambda_c,\lambda_c]$,
\[
\lim_{n \to \infty} X^n = X_\infty
\]
in law with respect to the topology of $\mc H^{-m}$ for
\[
m >
\left\{
\begin{array}{c@{\;;\;}l}
1/2 & d=1,\\
3 & d=2,\\
9/2 & d =3,
\end{array}
\right.
\]
where $X_\infty$ is a centered Gaussian process of variance given by
\begin{equation}
\label{caldera}
\bb E[X_\infty(f)^2] =  \sum_{k \in \bb Z^d} \big|\hat{f}(k)\big|^2 
\Big(\chi(\rast) + \frac{G(\rast) + 2 F'(\rast) \chi(\rast)}{8 \pi^2 \|k\|^2 - 2 F' (\rast)}\Big).
\end{equation}
for every $f \in \mc C^\infty(\bb T^d; \bb R)$.
\end{theorem}

\begin{remark}
Observe that if $\lambda =0$, then $X_\infty$ is a white noise of variance $\chi(\rast)$.
\end{remark}

\section{The relative entropy method}
\label{s3}
In this section we will prove Theorem \ref{t2} and we will use it to prove
Theorem \ref{t1}. We will use Yau's \emph{relative entropy method}, introduced
in \cite{Yau}; see Chapter 6 of \cite{KipLan} for a review. We will use the
approach of \cite{JarMen}.

\subsection{Yau's inequality}
\label{s3.1}
The \emph{carr\'e du champ} associated to the operator $L_n$ is the bilinear
operator $\Gamma_n$ given by
\[
\Gamma_n (f,g) := L_n(fg) - f L_n g - g L_n f
\]
for every $f,g : \omn \to \bb R$. As usual in functional analysis, we will use
the notation $\Gamma_n f:= \Gamma_n(f,f)$.
We will also define the \emph{carr\'es du champ} associated to the operators
$L_n^{\mathrm{ex}}$, $L_n^{\mathrm{r}}$:
\[
\Gamma_n^{\mathrm{ex}} (f,g) := L_n^{\mathrm{ex}}(fg) - fL_n^{\mathrm{ex}} g - g
L_n^{\mathrm{ex}}f,
\]
\[
\Gamma_n^{\mathrm{r}}(f,g) := L_n^{\mathrm{r}} (fg) -f
L_n^{\mathrm{r}} g - g L_n^{\mathrm{r}} f
\]
for every $f,g: \omn \to \bb R$. 

Let $L_n^\ast$ be the adjoint of $L_n$ with respect to $\mc L^2(\nurnast)$.
Observe that $\nurnast$ is reversible under $L_n^{\mathrm{ex}}$, that is,
$(L_n^{\mathrm{ex}})^\ast = L_n^{\mathrm{ex}}$. Therefore, $L_n^\ast = L_n^{\mathrm{ex}} +
(L_n^{\mathrm{r}})^\ast$. The operator $(L_n^{\mathrm{r}})^\ast$ can be computed
explicitly in terms of $\nurnast$:
\begin{equation}
\label{vallenar}
(L_n^{\mathrm{r}})^\ast f(\eta) = \sumx \Big( c_x(\eta^x)
\frac{\nurnast(\eta^x)}{\nurnast(\eta)}f(\eta^x) - c_x(\eta) f(\eta)\Big).
\end{equation}
A more explicit form can be obtained observing that
\[
\frac{\nurnast(\eta^x)}{\nurnast(\eta)} = \frac{\eta_x(1-\rast)}{\rast} +
\frac{(1-\eta_x)\rast}{1-\rast}.
\]
We will use this identity to compute $L_n^\ast
\mathbf{1}$, where $\mathbf{1}$ is the constant function equal to $1$.

Since we already know that the density of particles under $\muss$ is
approximately equal to $\rast$, it is reasonable to start the chain $(\etant; t
\geq 0)$ from the initial measure $\nurnast$. Let $f_t^n$ be the density with respect to $\nurnast$ of the
law of $\etant$ under $\bb P_{\nurnast}^n$. Let
\[
H_n(t) := H (f_t^n ; \nurnast )
\]
be the \emph{relative entropy} of the law of $\etant$ under $\bb P_{\nurnast}^n$
with respect to $\nurnast$. The so-called \emph{Yau's inequality} \cite{Yau, JarMen}, states that
\begin{equation}
\label{Yau}
H_n'(t) \leq - \int \Gamma_n \sqrt{f_t^n} d \nurnast + \int L_n^\ast \mathbf{1}
f_t^n d \nurnast.
\end{equation}
Since the chain $(\etant; t \geq 0)$ is irreducible and the state space $\Omega_n$ is finite,  $f_t^n \to \fuss$ as $t \to
\infty$, and in particular, since the function $x\to x\log(x)$ is continuous, we have 
\[
H( \fuss ; \nurnast) = \lim_{t \to \infty} H_n(t).
\]
Therefore, a uniform bound on $H_n(t)$ implies a bound on the relative entropy
of $\muss$ with respect to $\nurnast$. For every $\eta \in \Omega_n$ and every $x \in \tdn$, let us define
\[
\bareta_x := \eta_x - \rast.
\]
From \eqref{vallenar}, 
\begin{equation}
\label{incahuasi}
\begin{split}
 L_n^\ast \mathbf{1} 
    &= \sumx \Big\{\eta_x \Big( \Big( a+ \frac{\lambda}{2d} \sumyx \eta_y \Big)
\frac{1-\rast}{\rast} -b\Big) +(1-\eta_x) \Big( \frac{b \rast}{1-\rast}
- \Big(a+\frac{\lambda}{2d} \sumyx \eta_y\Big)\Big) \Big\}\\
    &= \sumx \Big( \frac{\eta_x}{\rast} - \frac{1-\eta_x}{1-\rast} \Big)
\Big(\Big(a+\frac{\lambda}{2d}\sumyx \eta_y\Big)(1-\rast) -b \rast \Big) \\
    &= \sumx \frac{\bareta_x}{\chi(\rast)} \Big( F(\rast) + \frac{\lambda}{2d} \sumyx \bareta_y (1-\rast)
\Big)
\\
    &= \frac{\lambda}{2d \rast}  \sumxy \bareta_x \bareta_y,
\end{split}
\end{equation}
where in the last identity we used that  $F(\rast) =0$. Observe that $L_n^\ast \mathbf{1}$
is a sum of monomials of degree (at least) two in the centered variables
$(\bareta_x ; x \in \tdn)$. The absence of monomials of degree $1$ in the expression for  $L_n^\ast \mathbf{1}$ is
fundamental in what follows.

\subsection{The log-Sobolev inequality}
\label{s3.2}
In order to take full advantage of Yau's inequality \eqref{Yau}, it would be useful to have a lower bound for the integral $\int \Gamma_n \sqrt{f_t^n} d \nurnast$ in terms of the relative entropy $H_n(t)$. This is exactly the content of the so-called \emph{log-Sobolev inequality}, which we now explain. Let $L$ be a Markov generator on $\Omega_n$ and let $\Gamma$ be the \emph{carr\'e du champ} associated to $L$. The \emph{log-Sobolev constant} of $L$ (or $\Gamma$) with respect to a measure $\mu$ in $\Omega_n$, is defined as
\begin{equation}
\label{logsoconst}
\alpha = \alpha ( \Gamma; \mu) := \inf \frac{\int \Gamma \sqrt f d \mu}{ H ( f; \mu)},
\end{equation}
where the infimum runs over all densities $f$ with respect to $\mu$. It can be shown that $\alpha$ is always finite and that if $L$ is irreducible, then $\alpha >0$. If $L$ is irreducible and $\mu$ is equal to the invariant measure of $L$, then $\alpha$ can be used to estimate the speed of convergence to $\mu$ of the law of the chain generated by $L$, see \cite{DiaS-C} for a review. 

For the generator $L_n$ defined in \eqref{tocopilla} and the measure $\nurnast$, the log-Sobolev constant is bounded from below by a positive constant that does not depend on $n$, a result known in the literature as the \emph{log-Sobolev inequality}:

\begin{lemma}[log-Sobolev inequality]
\label{logso}
For every $n \in \bb N$, every $\rho \in (0,1)$, every $a,b >0$ and every $\lambda >-a$,
\[
\alpha( \Gamma_n^{\mathrm{r}}; \nurn) \geq \frac{ \eps_0 |1-2\rho|}{2\rho(1-\rho)\big|\log \frac{\rho}{1-\rho}\big|}= : \kappa(\rho)^{-1},
\]
 where $\eps_0:= \min\{a,a+\lambda, b\}$.
\end{lemma}

\begin{remark}
By continuity, this lemma holds with right-hand side equal to $\eps_0$ when $\rho = 1/2$.
\end{remark}

\begin{proof}
This is a classical result in the literature, so we will only give a sketch of its proof. According to \cite[Theorem A.2 and Lemma 3.2]{DiaS-C}, 
\[
H( f;  \nurn) \leq \frac{2 \rho (1-\rho) \big| \log  \frac{\rho}{1-\rho}\big|}{|1-2\rho|} \sumx \int \big( \nabla_x \sqrt{f} \big)^2 d \nurn
\]
for every $\rho \in (0,1)$ and every density $f$ with respect to $\nurn$. Since $c_x(\eta) \geq \eps_0$ for every $x \in \tdn$ and every $\eta \in \omn$, 
\[
\int \Gamma^{\mathrm{r}}_n \sqrt{f} d \nurn \geq \eps_0 \sumx \int \big( \nabla_x \sqrt{f} \big)^2 d \nurn
\]
for every density $f$ with respect to $\nurn$, which proves the lemma.
\end{proof}

\subsection{The main lemma} 
\label{s3.3}
In order to make an effective use of Yau's inequality as stated in \eqref{Yau}, we need to
estimate the integral $\int L_n^\ast \mathbf{1} f_t^n d \nurnast$ in terms of the relative entropy
$H_n(t)$ and the \emph{energy} $\int \Gamma_n \sqrt{f_t^n} d \nurnast$. It turns out that the
specific form of $f_t^n$ as the density of a Markov chain
does not play a role in this estimation procedure. Therefore, we will estimate $\int L_n^\ast
\mathbf{1} f d \nurnast$ for arbitrary densities $f$. Moreover, it will be
useful to consider $L_n^\ast \mathbf{1}$ as a particular instance of the sum
\begin{equation}
\label{cachiyuyo}
V(g) := \sum_{i=1}^d \sumx \bareta_x \bareta_{x+e_i} g_x^i,
\end{equation}
where $g=(g^1,\dots,g^d): \tdn \to \bb R^d$ is a given function and where $\{e_1,\dots,e_d\}$
denotes the canonical basis of $\tdn$. According to \eqref{incahuasi}, $L_n^\ast \mathbf{1}$
corresponds to the case $g_x^i= \frac{\lambda}{2d \rast}$.
The so-called \emph{main lemma}, introduced in Theorem 3.1 of \cite{JarMen},
allows us to
replace $\bareta_x \bareta_y$ by products of averages in boxes of mesoscopic
size, with a cost controlled by the energy $\int \Gamma_n \sqrt{f_t^n} d \nurnast$.
Since we need to be very precise about the dependence on $\lambda$ of the
constants appearing in this lemma, we will present its proof. This proof follows very closely the proof in \cite{JarMen} and it can be
omitted in a first reading. Recall the definition of $\kappa(\rho)$ given in Lemma \ref{logso} and let us define $\mc A(u) := u(1+u)$ for every $u \geq 0$.

\begin{theorem}[Main lemma]
 \label{main}
There exists a constant $C= C(d)$ such that
\[
\begin{split}
\int V(g) f d \nurnast 
    &\leq \frac{1}{4} \int \Gamma_n^{\mathrm{ex}} \sqrt{f} d
\nurnast + C \kappa(\rast) \mc A(\|g\|_\infty) \int \Gamma_n^{\mathrm{r}} \sqrt{f} d
\nurnast \\
    &\quad+ C \mc A(\|g\|_\infty) n^{d-2} g_d(n)
\end{split}
\]
for every $n \in \bb N$, every $g: \tdn \to \bb R^d$ and every density $f$ with respect to $\nurnast$.\end{theorem}

\begin{proof}
For $\ell < n$ and $x \in \tdn$, let $\bb C_x^\ell$ the cube of vertex $x$ and
side $\ell$ given by
\[
\bb C_x^\ell := \{y \in \tdn; y_i-x_i \in \{0,1,\dots,\ell-1\} \text{ for every
} i \in \{1,\dots,d\}\}.
\]
Let $p^\ell:\tdn \to [0,1]$ be the uniform measure in $\bb C_0^\ell$, that is,
$p^\ell(x) = \ell^{-d}\mathbf{1}(x \in \bb C_0^\ell)$ for every $x \in \tdn$.
For $\ell <n/2$, let $q^\ell: \tdn \to [0,1]$ be given by $q^\ell := p^\ell \ast
p^\ell$, that is,
\[
q^\ell(y) := \sumx p^\ell(y-x) p^\ell(x)
\]
for every $y \in \tdn$. Observe that $q^\ell$ is supported on $\bb
C^{2\ell-1}_0$. 

For $\eta \in \omn$, $x \in \tdn$ and $\ell <n/2$, let
$\bareta_x^\ell \in \bb R$ be given by
\[
\bareta_x^\ell := \sumy q^\ell(y) \bareta_{x+y}.
\]
Let us define as well
\[
\vetax{\ell}{x} := \sumy p^\ell(y) \bareta_{x+y}, \quad
\vvetax{\ell,i}{x}(g) := \sumy p^\ell(y)\bareta_{x-y} g_{x-y}^i,
\]
\[
V^\ell(g) := \sum_{i=1}^d \sumx \bareta_x \bareta_{x+e_i}^\ell g_x^i.
\]
Thanks to our choice of the probability measure $q^\ell$, the sum $V^\ell(g)$ can be rewritten as
\[
V^\ell(g) = \sum_{i=1}^d \sumx \vvetax{\ell,i}{x}(g) \vetax{\ell}{x+e_i}.
\]
This identity will make the proof of Lemma \ref{main} shorter. 
The idea is to estimate $\int (V(g) - V^\ell (g))f d \nurnast$ in terms of $\int
\Gamma_n^{\mathrm{ex}} \sqrt f d \nurnast$. Observe that
\[
V(g) - V^\ell(g) =  \sum_{i=1}^d \sumx \bareta_x (\bareta_{x+e_i} -
\bareta_{x+e_i}^\ell) g_x^i.
\]
Let us rewrite the difference $\bareta_x - \bareta_x^\ell$ as a linear
combination of terms of the form $\bareta_z-\bareta_y$, where $y \sim z$. In
order to do that, we will use \emph{flows}.

Let $\mc E^d_n := \{(x,y); x, y \in \tdn, x \sim y\}$ be the set of oriented edges
of the periodic lattice $\tdn$. We say that a function $\phi: \mc E^d_n \to \bb
R$ is a \emph{flow} if $\phi(x,y) = - \phi(y,x)$ for every $(x,y) \in \mc
E^d_n$. Let $p,q$ be two probability measures in $\tdn$. We say that a flow
$\phi$ \emph{connects} $p$ to $q$ if
\[
p(x)-q(x) = \sumyx \phi(x,y)
\]
for every $x \in \tdn$. In other words, $p-q = \divergence \phi$. In that case,
$p$, $q$ and $\phi$ satisfy the \emph{divergence formula}: for every $g: \tdn
\to \bb R$,
\begin{equation}
\label{div}
\sumx g(x)(p(x)-q(x)) = \sumxy \phi(x,y) (g(x)-g(y)).
\end{equation}
Recall the definition of $g_d(n)$ given in \eqref{gd}. Lemma 3.2 in \cite{JarMen} tells us that for every $\ell < n/2$ there exists a
finite $C=C(d)$ and a flow $\phi^\ell$ connecting the Dirac $\delta$ at $x=0$
to $q^\ell$ such that
\begin{itemize}
 \item [i)] $\displaystyle{\sumxy \phi^\ell(x,y)^2 \leq C g_d(\ell)}$,

 \item[ii)] $\phi^\ell(x,y)=0$ whenever $x \notin \bb C_0^{2\ell-1}$ or $y
\notin \bb C_0^{2\ell-1}$.
\end{itemize}

Using the flow $\phi^\ell$ given by Lemma 3.2 of \cite{JarMen}, the translation
invariance of the lattice $\tdn$  and the divergence formula \eqref{div}, we see
that
\[
\bareta_x - \bareta_x^\ell = \sumyz
\phi^\ell(y,z) (\bareta_{x+y}-\bareta_{x+z}), 
\]
from where
\[
\begin{split}
 \sum_{i=1}^d \sumx \bareta_x (\bareta_{x+e_i} -
\bareta_{x+e_i}^\ell) g_x^i 
      &= \sum_{i=1}^d \sumx \;\; \sumyz \phi^\ell(y,z)
(\bareta_{x+y+e_i}-\bareta_{x+z+e_i}) \bareta_x g_x^i\\
      &= \sumyz \biggl(\sum_{i=1}^d \sumx \phi^\ell(y-x,z-x) \bareta_{x-e_i}
g_{x-e_i}^i \biggr) (\bareta_y-\bareta_z)\\
      &= \sumyz h_{y,z}^\ell(g;\eta) (\bareta_y- \bareta_z),
\end{split}
\]
where $h_{y,z}^\ell(g): \omn \to \bb R$ is defined as
\[
h_{y,z}^\ell(g;\eta) := \sum_{i=1}^d \sumx \phi^\ell(y-x,z-x) \bareta_{x-e_i}
g_{x-e_i}^i
\]
for every $\eta \in \omn$. Observe that $h_{y,z}^\ell(g)$ does not depend on $\eta_y$, $\eta_z$. 
From Lemma E.2 of \cite{JarMen}, we have that
\[
\int h_{y,z}^\ell(g) (\bareta_y- \bareta_z) f d\nurnast \leq \beta_0 \int
(\nabla_{y,z} \sqrt{f})^2 d \nurnast + \frac{1}{\beta_0} \int h_{y,z}^\ell(g)^2 f 
d \nurnast
\]
for every $\beta_0 > 0$. Choosing $\beta_0 = \beta n^2$ and taking the sum over $y \sim z$, we see that
\begin{equation}
\label{coquimbo}
\int (V(g)-V^\ell(g)) f d \nurnast \leq \beta \int \Gamma_n^{\mathrm{ex}}
\sqrt{f} d \nurnast + \frac{1}{2\beta n^2} \int \sumyz h_{y,z}^\ell(g)^2 f d
\nurnast.
\end{equation}
This estimate is exactly Lemma 3.3 of \cite{JarMen}. In our case, the proof is
simpler due to the translation invariance of the reference measure $\nurnast$.
Our next task is to take advantage of the average over boxes of size $\ell$ in
$V^\ell(g)$ in order to get a good estimate in terms of $H(f; \nurnast)$. 
The so-called \emph{entropy inequality} says that for every function
$h: \Omega_n \to \bb R$, every density $f$ with respect to $\nurnast$ and every
$\gamma > 0$,
\begin{equation}
\label{loa}
\int h f d \nurnast \leq \frac{1}{\gamma}\Big( H(f; \nurnast) + \log \int e^{\gamma h} d
\nurnast\Big).
\end{equation}
In order to be able to use this estimate, we need to compute the exponential
moments of the variables $h_{x,y}^\ell(g)^2$ and $\vvetax{\ell,i}{x}(g) \vetax{\ell}{x+e_i}$
with respect to $\nurnast$. Using independence and Lemma \ref{Hoeff}, we see
that for every $\theta \in \bb R$,
\begin{equation}
\label{huasco}
\begin{split}
 \log \int e^{\theta h_{y,z}^\ell(g)} d \nurnast
    &\leq \sumx \frac{\theta^2}{8} \Big( \sum_{i=1}^d
\phi^\ell(y-x-e_i,z-x-e_i) g_x^i \Big)^2 \\
    &\leq \frac{d \theta^2}{8} \|g\|_\infty^2 \sum_{i=1}^d \sumx
\phi^\ell(y-x-e_i,z-x-e_i)^2\\
    &\leq C(d) \theta^2 \|g\|_\infty^2 g_d(\ell).
\end{split}
\end{equation}
Therefore, by Lemma \ref{subG}, 
\begin{equation}
\label{vicuna}
\int e^{\gamma h_{y,z}^\ell(g)^2} d \nurnast 
    \leq \big(1- C(d) \gamma \|g\|_\infty^2
g_d(\ell)\big)^{-1/2}
\end{equation}
whenever $\gamma^{-1} > C(d) \|g\|_\infty^2
g_d(\ell)$.

Observe that $h_{y,z}^\ell(g)$ and $h_{y',z'}^\ell(g)$ are independent whenever
$|y+z-y'-z'| > 4 \ell$. Therefore, the set $\mc E_n^d$ can be divided into
$C(d) \ell^d$ disjoint sets $\{A^i; i \in \mc I\}$ such that the terms in the
sums
\[
\sum_{(y,z) \in A^i} h_{y,z}^\ell(g)^2
\]
are mutually independent. Therefore, for $\gamma > 0$ small enough,
\[
\begin{split}
 \int \sumyz h_{y,z}^\ell(g) ^2 f d \nurnast 
    &= \sum_{i \in \mc I} \int \sum_{(y,z) \in A^i} h_{y,z}^\ell(g)^2 f  d
\nurnast \\
    &\leq \sum_{i \in \mc I} \frac{1}{\gamma} \Big( H(f; \nurnast) + \log \int
e^{\gamma\sum_{(y,z) \in A^i} h_{y,z}^\ell(g)^2 } d \nurnast\Big)\\
    &\leq \sum_{i \in \mc I} \frac{1}{\gamma} \Big(H(f; \nurnast) +\sum_{(y,z) \in A^i}\log
\int e^{\gamma h_{y,z}^\ell(g)^2 } d \nurnast\Big)\\
    &\leq \gamma^{-1}C(d) \ell^d H(f; \nurnast) + \gamma^{-1} n^d \log \big(1- \gamma C(d)
\|g\|_\infty^2 g_d(\ell) \big)^{-1/2}.
\end{split}
\]
In the last line we used \eqref{vicuna}. Taking $\gamma^{-1}=2 C(d) \|g\|_\infty^2 g_d(\ell)$ we conclude that
\begin{equation}
\label{laserena}
\int \sumyz h_{y,z}^\ell(g)^2 f d \nurnast \leq C(d) \|g\|_\infty^2 \ell^d
g_d(\ell) \Big( H(f; \nurnast) + \frac{n^d}{\ell^d}\Big).
\end{equation}
Repeating the arguments in \eqref{huasco} for the random variables
$\vvetax{\ell,i}{x}(g)$ and $\vetax{\ell}{x}$ for each $x \in \tdn$ and each $i \in \{1,\dots,d\}$, we see that for every $\theta \in
\bb R$,
\[
\log \int e^{\theta \vetax{\ell}{x}} d \nurnast \leq \sum_{y \in \tdn}
\tfrac{1}{8} p^\ell(y)^2 \theta^2 \leq \frac{\theta^2}{8 \ell^d}
\]
and
\[
\log  \int e^{\theta \vvetax{\ell,i}{x}(g)} d \nurnast \leq
\frac{\|g\|_\infty^2 \theta^2}{8 \ell^d}.
\]
Since $\vvetax{\ell,i}{x}(g)$ and  $\vetax{\ell}{x+e_i}$ are independent under
$\nurnast$, using Lemma \ref{subG} we obtain the bound
\[
\int e^{\gamma \vvetax{\ell,i}{x}(g) \vetax{\ell}{x+e_i}} d \nurnast
    \leq \int \exp\Big\{\frac{\|g\|_\infty^2 \gamma^2}{8 \ell^d}\big(
\vetax{\ell}{x+e_i}\big)^2\Big\} d \nurnast
    \leq \Big(1 - \frac{\gamma^2 \|g\|_\infty^2}{16 \ell^{2d}}\Big)^{-1/2} 
\]
for every $\gamma < 4 \ell^d \|g\|_\infty^{-1}$. 
Observe that for $\gamma =\frac{2 \ell^d}{\|g\|_\infty}$, we obtain the bound
\[
\int e^{\gamma \vvetax{\ell,i}{x}(g) \vetax{\ell}{x+e_i}} d \nurnast \leq \frac{2}{\sqrt 3}.
\]
Observe that the random variables $\vvetax{\ell,i}{x}(g) \vetax{\ell}{x+e_i}$ and 
$\vvetax{\ell,j}{y}(g) \vetax{\ell}{y+e_j}$ are independent under $\nurnast$ as
soon as $\|y-x\| > 2\ell$. Dividing the set $\mc E_n^d$ as above, we obtain the
estimate
\begin{equation}
\label{tongoy}
\int \sumx \sum_{i=1}^d \vvetax{\ell,i}{x}(g) \vetax{\ell}{x+e_i} f d \nurnast
\leq C(d) \|g\|_\infty \Big( H(f; \nurnast) + \frac{n^d}{\ell^d}\Big).
\end{equation}
Putting estimates \eqref{coquimbo}, \eqref{laserena} and \eqref{tongoy}
together, we conclude that for every $\beta >0$,
\begin{equation}
\label{ML}
\begin{split}
\int \sumx  \sum_{i=1}^d \bareta_x \bareta_{x+e_i} g_x^i f d \nurnast 
    &\leq \beta \int \Gamma_n^{\mathrm{ex}} \sqrt{f} d \nurnast\\
    &\quad + C(d) \Big(
\|g\|_\infty + \frac{\|g\|_\infty^2 \ell^d g_d(\ell)}{\beta n^2}\Big) \Big( H(f; \nurnast)
+ \frac{n^d}{\ell^d}\Big).
\end{split}
\end{equation}
Choosing in this estimate $\beta = \frac{1}{4}$ and 
\[
\ell = 
\left\{
\begin{array}{c@{\;;\;}l}
 \frac{n}{4} & d=1,\\
 \frac{n}{\sqrt{\log n}} & d=2,\\
 n^{d-2} & d \geq 3,
\end{array}
\right.
\]
we see that
\begin{equation}
\label{losvilos}
\int \sumx  \sum_{i=1}^d \bareta_x \bareta_{x+e_i} g_x^i f d \nurnast 
		\leq \frac{1}{4} \int \Gamma_n^{\mathrm{ex}} \sqrt{f} d \nurnast 
			+C(d) \mc A(\|g\|_\infty) \big( H(f; \nurnast) + n^{d-2} g_d(n)\big).
\end{equation}
By the definition in \eqref{logsoconst} and Lemma \ref{logso},
\[
H(f; \nurnast) \leq \kappa(\rast) \int \Gamma_n^ {\mathrm{r}} \sqrt{f} d \nurnast
\]
for every density $f$. Therefore,
\[
\begin{split}
\int \sumx  \sum_{i=1}^d \bareta_x \bareta_{x+e_i} g_x^i f d \nurnast 
		&\leq \frac{1}{4} \int \Gamma_n^{\mathrm{ex}} \sqrt{f} d \nurnast
			+ C \kappa(\rast) \mc A(\|g\|_\infty) \int \Gamma_n^{\mathrm{r}} \sqrt{f} d \nurnast\\
		&\quad + C \mc A(\|g\|_\infty) n^{d-2} g_d(n),
\end{split}
\]
which proves the lemma.
\end{proof}

\subsection{Proof of Theorem \ref{t2}}
\label{s3.4}
Using Theorem \ref{main}, it is not difficult to carry out the proof of Theorem \ref{t2}. If we take $g_x^ i = \frac{\lambda}{2 d \rast}$ for every $x \in \tdn$ and every $i \in \{1,\dots,d\}$, if we take $f = f_t^ n$ in Theorem \ref{main} and if we put the resulting estimate into \eqref{Yau}, then we see that
\[
H_n' (t) \leq - \frac{3}{4} \int \Gamma_n^ {\mathrm{ex}} \sqrt{f_t^n} d \nurnast + \big( C \kappa(\rast) \mc A \Big( \frac{|\lambda|}{d \rast}\Big) - 1 \big) \int \Gamma_n^ {\mathrm{r}} \sqrt{f_t^ n} d \nurnast + C \mc A\Big( \frac{|\lambda|}{d \rast} \Big) n^{d-2} g_d(n).
\]
Observe that
\[
\lim_{\lambda \to 0} \eps_0 = \min\{a,b\}, \; \lim_{\lambda \to 0} \rast = \frac{a}{a+b} \text{ and } \lim_{\lambda \to 0} \kappa(\rast) = \frac{|b^ 2 - a^ 2| \min\{a,b\}}{2 ab |\log \frac{b}{a}|}.
\]
Therefore,
\[
\lim_{\lambda \to 0} \big( C\kappa(\rast) \mc A \Big(\frac{|\lambda|}{d \rast} \Big) -1 \big) =-1
\]
and there exists $\lambda_c>0$ such that 
\begin{equation}
\label{laligua}
 C \kappa(\rast) \mc A\Big( \frac{\lambda}{d \rast}\Big) < \frac{1}{2} \text{ and } \mc A \Big( \frac{\lambda}{d \rast} \Big) \leq 1
\end{equation}
for every $\lambda \in [-\lambda_c,\lambda_c]$. Using Lemma \ref{logso}, we see that for $\lambda \in [-\lambda_c,\lambda_c]$,

\[
\begin{split}
H_n'(t) 
	& \leq - \frac{1}{2}  \int \Gamma_n^ {\mathrm{r}} \sqrt{f_t^ n} d \nurnast + C \mc A \Big( \frac{|\lambda|}{d \rast} \Big) n^{d-2} g_d(n)\\
	& \leq -\frac{1}{2 \kappa(\rast)} H_n(t) + C n^{d-2} g_d(n).
\end{split}
\]
Using $e^{\frac{t}{2 \kappa(\rast)}}$ as an integrating factor, since $H_n(0) =0$ we conclude that
\[
H( f_t^n; \nurnast) = H_n(t) \leq C n^{d-2} g_d(n)
\]
for every $t \geq 0$. Taking $t \to \infty$, Theorem \ref{t2} is proved.

\subsection{Quantitative hydrostatics}
\label{s3.5}
In this section we prove Theorem \ref{t1} as a corollary of Theorem \ref{t2}. It will be useful to prove a more general version of Theorem \ref{t1}. Recall the definitions introduced before Theorem \ref{loceq}.
We say that a function $\psi: \Omega_n \to \bb R$ has \emph{support in} $B_R$ if $\psi(\eta^x) = \psi(\eta)$ for every $x \notin \mc P(B_R)$ and every $\eta \in \Omega_n$. For $x \in \tdn$ and $\eta \in \Omega_n$, let $\tau_x \eta \in \Omega_n$ be the translation of $\eta$ by $x$, that is, $(\tau_x \eta)_z = \eta_{z+x}$ for every $z \in \tdn$. 
For $g: \Omega_n \to \bb R$ and $x \in \tdn$, let $\tau_x g: \Omega_n \to \bb R$ be the translation of $g$ by $x$, that is, $\tau_x g(\eta) := g(\tau_x \eta)$ for every $\eta \in \Omega_n$. 
For $\psi$ with support in $B_R$ and $x \in \tdn$, let us write $\psi_x: = \tau_x \psi$. Define $\<\psi\>_\rast := \int \psi d \nurnast$. For each $f \in \mc C(\bb T^d; \bb R)$, let $\Psi^n(f) : \Omega_n \to \bb R$ be given by
\begin{equation}
\label{lascruces}
\Psi^n(f; \eta) := \frac{1}{n^d} \sum_{x \in \tdn} \big( \psi_x(\eta) - \< \psi \>_\rast\big) f \big( \tfrac{x}{n} \big)
\end{equation}
for every $\eta \in \Omega_n$. Theorem \ref{t1} is a particular case of the following estimate:

\begin{theorem}
\label{t4}
There exists a finite constant $C = C(\psi; \rast)$ such that
\[
\int \Psi^n(f)^2 d \muss \leq \frac{C \|f\|_{\ell^2_n}^2 g_d(n)}{n^2}
\]
for every $f \in \mc C(\bb T^d; \bb R)$, where
\[
\|f\|_{\ell^2_n}^2 := \frac{1}{n^d} \sum_{x \in \tdn} f \big( \tfrac{x}{n} \big)^2.
\]
\end{theorem}

\begin{proof}
Let $\Osc(\psi):= \sup_\eta \psi(\eta) - \inf_\eta \psi(\eta)$. By Hoeffding's Lemma, for every $x \in \tdn$ we have that
\[
\log \int e^{ \theta ( \psi_x -\<\psi\>_\rast)} d \nurnast \leq \frac{\theta^2 \Osc(\psi)^2}{8}.
\]
Since $B_R$ is a cube of side $2R+1$, $\psi_x$ and $\psi_y$ are independent as soon as $\|y-x \|_\infty > 2R+1$. Therefore, $\tdn$ can be split into disjoint sets $(A^i; i \in \mc I_R)$ such that the cardinality of $\mc I_R$ is at most $C(d)(2R+1)^d$  and such that the terms in the sums
\[
\frac{1}{n^d} \sum_{x \in A^i} \big(\psi_x - \<\psi\>_{\rast}\big) f \big(\tfrac{x}{n} \big)
\]
are all independent.
Therefore,
\[
\begin{split}
\log \int e^{\theta \Psi^n(f) } d \nurnast 
	&\leq \tfrac{C(d)}{(2R+1)^d} \sum_{i \in \mc I_R} \log \int \exp \big\{ \tfrac{\theta C(d)(2R+1)^d}{ n^{d}} \sum_{x \in A^i} (\psi_x- \<\psi\>_\rast) f \big( \tfrac{x}{n} \big)\big\} d \nurnast\\
	&\leq \tfrac{C(d)}{(2R+1)^d} \sum_{i \in \mc I_R} \sum_{x \in A^i} \tfrac{\theta^2 C(d) (2R+1)^{2d} \Osc(\psi)^2}{8 n^{2d}} f\big( \tfrac{x}{n} \big)^2 \\
	&\leq \frac{\theta^2 C(d) (2R+1)^d \Osc(\psi)^2\|f\|_{\ell^2_n}^2}{n^d}.
\end{split}
\]
By Lemma \ref{subG} we see that
\[
\int e^{\theta \Psi^ n(f)^2} d \nurnast \leq \sqrt 2
\]
for 
\[
\theta = \frac{8 n^d}{C(d)(2R+1)^d \|f\|_{\ell^2_n}^2\Osc(\psi)^2}.
\]
Using the entropy inequality \eqref{loa}, we see that
\begin{equation}
\label{quintero}
\begin{split}
\int \Psi^ n(f)^ 2 d \muss 
	&\leq \theta^{-1} \Big( H(\muss| \nurnast) + \log \int e^{\theta \Psi^ n(f)^2} d \nurnast \Big)\\
	&\leq \frac{C(d)(2R+1)^d \|f\|_{\ell^2_n}^2 \Osc(\psi)^2}{n^d} \big( C n^{d-2} g_d(n) + \log \sqrt 2 \big)\\
	&\leq \frac{C(\psi, \rast,d) \|f\|_{\ell^2_n}^2 g_d(n)}{n^2},
\end{split}
\end{equation}
as we wanted to show.
\end{proof}

Observe that Theorem \ref{t1} corresponds to Theorem \ref{t4} in the particular case $\psi(\eta) = \eta_0$. It turns out that Theorem \ref{t4} can also be used to prove Theorem \ref{loceq}.

\begin{proof}[Proof of Theorem \ref{loceq}]
In general, the  bound on the relative entropy given by Theorem \ref{t2} is not enough to derive a result as strong as Theorem \ref{loceq}. However, in our particular setting, we can take advantage of the translation invariance of the dynamics in order to improve the bounds. Observe that the generator of $(\eta(t); t \geq 0)$ satisfies
\[
\tau_x (L_nf)  = L_n (\tau_x f) 
\]
for every $f: \Omega_n \to \bb R$ and every $x \in \tdn$. In particular, if the law of $\eta(0)$ is translation invariant, then the law of $\eta(t)$ is translation invariant for every $t \geq 0$. Taking $t \to \infty$, we deduce that $\muss$ is translation invariant.
 
 Recall the definition of $\Pi_R$ given before Theorem \ref{loceq}.
For $R <\frac{n-1}{2}$ and $x \in \tdn$, let $\Pi_R^x$ be the canonical projection onto the box of size $R$ and center $x$, that is, $\Pi_R^x \eta := \Pi_R( \tau_x \eta)$ for every $\eta \in \Omega_n$. Observe that for every $\psi: \{0,1\}^{B_R} \to \bb R$ and every $x,y \in \tdn$, 
\[
\int \psi(\Pi_R^x \eta)  d \muss  = \int \psi (\Pi_R^y \eta) d \muss.
\]
Therefore, 
\[
\int \psi \,d \mussr{R} = \int \psi(\Pi_R^x \eta)  d \muss
\]
for every $x \in \tdn$. Recall the definition of $\Psi^n$ given in \eqref{lascruces}. Applying estimate \eqref{quintero} to the constant function $\mathbf{1}$, we see that
\[
\begin{split}
\Big| \int \psi \, d \mussr{R} - \int \psi \, d \nurnastr{R} \Big|^2
		&= \Big| \int \frac{1}{n^ d} \sum_{x \in \tdn} \big( \psi_x - \<\psi\>_\rast) d \muss \Big|^2\\
		&= \Big| \int \Psi^n(\mathbf 1) d \muss\Big|^2 \leq \frac{C (2R+1)^d g_d(n)\Osc(\psi)^ 2}{n^2}.
\end{split}
\]
Observe that $\|\psi\|_\infty \leq 1$ implies $\Osc(\psi) \leq 2$. Taking the supremum over $\|\psi\|_\infty \leq 1$ in this estimate and recalling the definition of $g_d(n)$, Theorem \ref{loceq} is proved.
\end{proof}

\section{Density fluctuations}
\label{s4}
\subsection{Case $d=1$}
In this section we prove Theorem \ref{t3} in the case $d=1$. In that case the proof is simpler, due to the fact that the estimate in Theorem \ref{t2} is uniform in $n$. This will allow us to explain better the ideas behind the proof, which will be used for $d=2,3$ afterwards. Some of the computations are not dimension dependent; in those cases we will keep the dependence on $d$ in the notation.

By Theorem \ref{t1}, for every $\lambda \in [-\lambda_c,\lambda_c]$, 
\[
\int X^n(f)^2 d \muss \leq C \|f\|_{\ell^2_n}^2.
\]
Recall definition \eqref{calama}. For $f(x) = \cos(2 \pi k x)$ or $f = \sin(2 \pi k x)$, $\|f\|_{\ell^2_n} \leq 1$. Therefore, for $m < -1/2$,
\[
\int \|X^n\|_{\mc H^m}^2 d \muss = \sum_{k \in \bb Z} (1+k^2)^m \int |\hat{X}^n(k)|^2 d \muss \leq C \sum_{k \in \bb Z} (1+k^2)^m \leq C(m) <+\infty.
\]
Recall that the inclusion $\mc H^{m'} \subseteq \mc H^{m}$ is compact if $m < m'$, and therefore balls in $\mc H^{-m'}$ are compact in $\mc H^m$. Therefore, for $m' \in (m,-1/2)$ the set $\{ \|X\|_{\mc H^{m'}} \leq M\}$ is compact in $\mc H^m$. From the previous estimate we see that
\[
\muss\big( \|X^n\|_{\mc H^{m'}} > M\big) \leq \frac{C(m')}{M^2}.
\]
Taking $M \to \infty$, we conclude that the sequence $\{X^n; n \in \bb N\}$ is tight in $\mc H^{m}$. 

The idea is to prove a limit theorem for the \emph{dynamical} fluctuation field, and to derive Theorem \ref{t3} from such result. Let $(X^n_t; t \geq 0)$ be the process given by
\begin{equation}
\label{fluctuations}
X_t^n(f) := X^n(\eta^n(t),f) = \frac{1}{n^ {d/2}} \sum_{x \in \tdn} (\eta_x^n(t) - \rast) f \big( \tfrac{x}{n} \big) 
\end{equation}
for every $f \in \mc C^\infty(\bb T^d; \bb R)$. The following result is a version of \cite[Theorem 2.4]{JarMen} for the model considered here, and it holds for dimensions $d \leq 3$:

\begin{proposition}
\label{noneq}
Let $(\mu_0^n; n \in \bb N)$ be a sequence of probability measures on $\Omega_n$ such that:
\begin{itemize}
\item
[i)] $\displaystyle{\sup_{n \in \bb N} H(\tfrac{d \mu_0^n}{d \nurnast} ; \nurnast) <+\infty}$, 

\item
[ii)] $X^n \to X_0$ in $\mc H^{-m}$ in law with respect to $\mu_0^n$ for some $m > d/2$.
\end{itemize}
For $d \leq 3$, the sequence $(X_t^n; t \geq 0)_{n \in \bb N}$ converges in the sense of finite-dimensional distributions to the process $(X_t; t \geq 0)$, solution of the equation
\begin{equation}
\label{SPDE}
\partial_t X_t = \Delta X_t + F' (\rast) X_t + \sqrt{2\chi(\rast)} \nabla \cdot \dot{\mc W}^1 + \sqrt{G(\rast)} \dot{\mc W}^2,
\end{equation}
with initial condition $X_0$, where $\dot{\mc W}^1$ is an $\bb R^d $-valued white noise, $\dot{\mc W}^2$ is a real-valued white noise and $\dot{\mc W}^i$, $i=1,2$ are independent.
\end{proposition}

We start proving a lemma:

\begin{lemma}
\label{l3.1}
For every dimension $d \geq 1$ and every $m > d/2$, the equation \eqref{SPDE} has a unique stationary solution supported in $\mc H^{-m}$.
\end{lemma}

\begin{proof}
Let $(P_t; t \geq 0)$ be the semigroup generated by the operator $\Delta +F'(\rast)$. For every $f \in \mc C^\infty(\bb T^d; \bb R)$ and every $t \geq 0$,
\begin{equation}
\label{santiago}
X_t(f) = X_0(P_t f) + \int_0^t \sqrt{2 \chi(\rast)} d \mc W_s^1(P_{t-s} \nabla f) + \int_0^t \sqrt{G(\rast)} d \mc W_s^2(P_{t-s} f).
\end{equation}
Since $F'(\rast)<0$, we see that $P_t f \to 0$ as $t \to \infty$ exponentially fast in $\mc H^m$ for every $m \in \bb R$. In consequence, 
the random family $(X_\infty(f); f \in \mc C^\infty(\bb T^d; \bb R))$ given by
\[
X_\infty(f) := \lim_{t \to \infty} X_t(f)
\]
for every $f \in \mc C^\infty(\bb T^d; \bb R)$ is well defined in law. Moreover, $X_\infty$ is a well-defined, $\mc H^{-m}$-valued random variable for $m$ large enough, and it satisfies the identity
\begin{equation}
\label{lacalera}
X_\infty(f)  = \int_0^\infty \Big( \sqrt{2 \chi(\rast)} d \mc W_t^1(P_{t} \nabla f) +\sqrt{G(\rast)} d \mc W_t^2(P_{t} f)\Big), 
\end{equation}
where this identity is understood as an identity in law for $\mc H^{-m}$-valued random variables. We claim that $X_\infty$ is the unique stationary state of \eqref{SPDE}. In fact, if there is another stationary state $Y_\infty$, the solution of \eqref{SPDE} with initial condition $Y$ on one hand satisfies $X_t = Y$ in law for every $t \geq 0$ and on the other hand it converges to $X_\infty$. 

The right-hand side of \eqref{lacalera} is a centered Gaussian process. Therefore, it is characterized by its covariance operator. Observe that
\[
\bb E[X_\infty(f)^2] = 2 \chi(\rast) \int_0^\infty \|P_t \nabla f\|^2_{ L^2} dt + G(\rast) \int_0^\infty \|P_t f\|^2_{ L^2} dt.
\]
In Fourier space, 
\[
\widehat{P_t f}(k)= e^{ - 4\pi^2 \|k\|^2 t + F'(\rast) t} \hat{f}(k).
\]
Therefore, by Parseval's identity,
\begin{equation}
\label{sanfelipe}
\bb E[X_\infty(f)^2] = \sum_{k \in \bb Z^d} \big|\hat{f}(k)\big|^2 \Big( \frac{4 \pi^2 \|k\|^2 \chi(\rast)}{4 \pi^2 \|k\|^2-F'(\rast)} + \frac{G(\rast)}{8 \pi^2 \|k\|^2 -2 F'(\rast)}\Big),
\end{equation}
which is equal to the expression given in \eqref{caldera} for the limit process $X$. Observe that
\begin{equation}
\label{llullaillaco}
\lim_{\|k\| \to \infty} \Big( \frac{4 \pi^2 \|k\|^2 \chi(\rast)}{4 \pi^2 \|k\|^2-F'(\rast)} + \frac{G(\rast)}{8 \pi^2 \|k\|^2 -2 F'(\rast)}\Big) = \chi(\rast),
\end{equation}
and in particular the coefficients in \eqref{sanfelipe} are bounded in $k$. Therefore, $\bb E[ \|X_\infty\|_{\mc H^{-m}}^2] <+\infty$ for every $m >d/2$ and in particular $X_\infty$ is a well-defined random variable in $\mc H^{-m}$ for every $m > d/2$.
\end{proof}
Let us go back to the proof of Theorem \ref{t3}. We have already proved that the sequence $(X^n; n \in \bb N)$ is tight in $\mc H^{-m}$ with respect to $\muss$, for $m > 1/2$. Let $n'$ be a subsequence for which $X^{n'}$ converges in law with respect to the topology of $\mc H^{-m}$ to a random variable $\widetilde{X}$. By Theorem \ref{t2}, the measures $\musspr$ satisfy the conditions of Proposition \ref{noneq}. Since $\musspr$ is stationary, $\widetilde{X}$ is a stationary solution of \eqref{SPDE}. By Lemma \ref{l3.1}, $\widetilde{X} = X_\infty$. We have just showed that $(X^n; n \in \bb N)$ is relatively compact and it has a unique accumulation point $X_\infty$. We conclude that $X^n$ converges to $X_\infty$, as we wanted to show.

\label{d1}

\subsection{Case $d=2,3$}

In this section we prove Theorem \ref{t3} in the case $d=2,3$.
In this case, the proof of Theorem \ref{t3} presented in the previous section does not work, because the entropy bound of Theorem \ref{t2} is not strong enough to derive tightness of the sequence $(X^n; n \in \bb N)$ with respect to $(\muss ; n \in \bb N)$.  The idea is to adapt the proof of Lemma \ref{l3.1}, \emph{directly} to the process $(X_t^n; t \geq 0)$ defined in \eqref{fluctuations}. In order to do that, let us introduce the Dynkin's martingales associated to $(X_t^n; t \geq 0)$. For every $T >0$ and every $g \in \mc C^{1,\infty}( [0,T] \times \bb T^d ; \bb R)$, the process $(\mc M_t^n(g); t \in [0,T])$ given by
\begin{equation}
\label{machali}
\mc M_t^n(g) := X_t^n(g_t) - X_0^n(g_0) - \int_0^t (\partial_s+L_n) X_s^n(g_s) ds
\end{equation}
for every $t \in [0,T]$, is a martingale of quadratic variation
\begin{equation}
\label{limache}
\< \mc M^n(g)\>_t = \int_0^t \Gamma_n X_s^n(g) ds.
\end{equation}
Our aim is to choose $g_s$ in such a way that equations \eqref{machali}, \eqref{limache} are approximate versions of \eqref{santiago}.
After some explicit computations, we see that 
\begin{equation}
\label{olmue}
(\partial_s+L_n) X^n(g_s) = X^n((\partial_s + \bb L_n)g_s) + \frac{1}{n^{d/2}} V( \vec g_{s})
\end{equation}
where the linear operator $\bb L_n$ is defined as
\[
\bb L_n f \big( \tfrac{x}{n} \big) := \Big( n^2 + \frac{\lambda(1-\rast)}{2d} \Big) \sumxy \, \big( f \big( \tfrac{y}{n} \big) - f \big( \tfrac{x}{n} \big) \big) + F'(\rast) f \big( \tfrac{x}{n} \big)
\]
for every $f \in \mc C(\bb T^d; \bb R)$ and every $x \in \bb T^d_n$, where $V$ was defined in \eqref{cachiyuyo} and where $\vec g_{s}$ is defined as
\[
g^i_{x,s} := -\frac{\lambda}{2 d} \big( g_{s} \big( \tfrac{x}{n} \big) + g_{s} \big(\tfrac{x+e_i}{n}\big)\big)
\]
for every $i \in \{1,\dots,d\}$ and every $x \in \tdn$.
We also see that
\begin{equation}
\label{quillota}
\Gamma_n X^n(g_s) = \frac{1}{2 n^d} \sumxy  (\eta_y - \eta_x)^2 n^2 \big( g_s\big( \tfrac{y}{n} \big) - g_s\big( \tfrac{x}{n} \big) \big)^2 
		+ \frac{1}{n^d} \sumx c_x(\eta) g_s\big( \tfrac{x}{n} \big)^2.
\end{equation}
Observe that $\bb L_n$ is a discrete approximation of the operator $\bb L:= \Delta + F'(\rast)$ used in the proof of Lemma \ref{l3.1}. For $f \in \mc C^\infty(\bb T^d; \bb R)$ and $T >0$, let $(g_{t,T} ; T >0, t \in [0,T])$ be given by $g_{t,T} := P_{T-t} f$ for every $T >0$ and every $t \in [0,T]$, where $(P_t; t \geq 0)$ is the semigroup generated by $\Delta + F'(\rast)$. For $T>0$ and $t \in [0,T]$, let $\mc M_{t,T}^n(f):= \mc M_t^n(g_{\cdot,T})$. Observe that $(\mc M_{t,T}^n(f); t \in [0,T])$ is a martingale and observe that
\begin{equation}
\label{quilpue}
\begin{split}
X_T^n(f) 
	&= X_0^n(g_{0,T}) + \int_0^T (\partial_t + L_n) X_t^n(g_{t,T}) dt + \mc M_{T,T}^n(f)\\
	&= X_0^n(g_{0,T}) + \int_0^T X_t^n((\partial_t + \bb L_n) g_{t,T}) dt + I_T^n(V;f) + \mc M_{T,T}^n(f),
\end{split}
\end{equation}
where
\[
I_T^n(V; f) :=  \int_0^T \frac{1}{n^{d/2}} V(\vec{g}_{t,T}) dt.
\]
Assuming that $X^n_0$ has law $\muss$, this equation is the discrete version of \eqref{cachiyuyo} we are aiming for. The proof of Theorem \ref{t3} will follow from a careful analysis of equation \eqref{quilpue}. This analysis will be divided into three parts. In Section \ref{s4.2.1} we will derived various estimates that will then be used in Section \ref{s4.2.2} to prove tightness of the field $X^n$ and in Section \ref{s4.2.3} to prove convergence of finite-dimensional laws of $X^n$.

\subsection{Estimates}
\label{s4.2.1}
Our aim is to prove estimates in $L^2(\bb P_\muss^n)$ of each of the terms appearing in \eqref{quilpue}, except $I_T^n(V; f)$, which we can only estimate in $L^1(\bb P_\muss^n)$. Since $F' (\rast)<0$, the semigroup $(P_t; t \geq 0)$ converges to 0 with respect to the uniform norm with exponential rate $-F' (\rast)$. We start estimating the second moment of $\bb E_{\muss}[X_0^n(g_{0,T})]$.
From Theorem \ref{t1},
\begin{equation}
\label{est1}
\begin{split}
\bb E^n_{\muss}[ X_0^n(g_{0,T} )^2] 
		&\leq \frac{C g_d(n)}{n^2} \sumx P_T f \big(\tfrac{x}{n} \big)^2 \\
		&\leq C g_d(n) n^{d-2} \|f\|_\infty^2 e^{2F'(\rast) T}.
\end{split}
\end{equation}
In order to estimate the second moment of $\int_0^T X_t^n((\partial_t + \bb L_n) g_{t,T}) dt$, observe that $(\partial_t + \bb L)g_{t,T}=0$. Therefore, by Taylor's formula,
\[
\big|(\partial_t + \bb L_n) g_{t,T}\big( \tfrac{x}{n} \big) \big| = \big|(\bb L_n - \bb L) g_{t,T}\big( \tfrac{x}{n} \big) \big|  \leq \frac{1} {12 n^2} \big( 6 \lambda(1-\rast) \|D^2 g_{t,T}\|_\infty +d\| D^4 g_{t,T}\|_\infty\big).
\]
Since the differential operator $D$ commutes with $\bb L$, we obtain the bound
\[
\big|(\partial_t + \bb L_n) g_{t,T}\big( \tfrac{x}{n} \big) \big| 
		\leq \frac{C} {12 n^2} \big(  \|D^2 f\|_\infty +\| D^4f\|_\infty\big) e^{F'(\rast)(T-t)}.
\]
Combining this bound with Theorem \ref{t1} and the invariance of $\muss$, we see that for every $\delta \in (0,-F' (\rast))$ and every $\lambda \in [-\lambda_c,\lambda_c]$,
\begin{equation}
\label{est2}
\begin{split}
\bb E_{\muss}^n \Big[\Big| \int_0^T X_t^n((\partial_t + \bb L_n)g_{t,T}) dt \Big|^2\Big] 
		&\leq \int_0^T e^{-2\delta t} dt  \int_0^T e^{2 \delta t} \bb E_{\muss}^n \big[ \big| X_t^n((\partial_t + \bb L_n)g_{t,T}) \big|^2 \big] dt\\
		& \leq \frac{C}{2 \delta} \int_0^T  g_d(n) n^{d-6} (\|D^2 f\|_\infty^2 + \| D^4 f\|_\infty^2) e^{2(\delta+F'(\rast)) t} dt\\
		& \leq C g_d(n) n^{d-6}  (\|D^2 f\|_\infty^2 + \| D^4 f\|_\infty^2).
\end{split}
\end{equation}
The martingale term is easy to estimate. We have that
\begin{equation}
\label{est3}
\begin{split}
\bb E_{\muss}^n[\mc M_{T,T}^n(f)^2] 
		&= \bb E_{\muss}^n \Big[ \int_0^T \Gamma_n X_t^n(g_{t,T})dt \Big] \leq \int_0^T C\big( \| D g_{t,T}\|_\infty^2 + \|g_{t,T}\|_\infty^2\Big) dt\\
		& \leq C\big( \|D f\|_\infty^2 + \| f\|_\infty^2\Big).
\end{split}
\end{equation}
The expectation
$
\bb E_{\muss}^n [ | I_T^n(V; f) | ]
$
is the most difficult to estimate. In fact, we are only able to obtain a first-moment estimate. By the entropy inequality \eqref{loa} and
$
H( \bb P_{\muss}^n | \bb P_{\nurnast}^n )  = H( \muss| \nurnast)
$ (which is a consequence of   Markov's property) we have that
\begin{equation}
\label{sanantonio}
\bb E_{\muss}^n [ | I_T^n(V; f) | ]
		\leq  \frac{1}{\theta} \Big( H(\muss| \nurnast) \!+\! \log \bb E^n_{\nurnast} \big[ \exp \big\{ \theta | I_T^n(V; f) | \big\}\big] \Big).
\end{equation}
Using the elementary bounds $\exp|x|\leq \exp (x)+\exp (-x)$, $\log(a+b) \leq \log 2 + \max\{\log a, \log b\}$ and for the choice 
\[
\theta:= \frac{\lambda_c n^{d/2}}{2 d \rast \sup_{t,T} \|\vec g_{t,T}\|_\infty},
\]
we see that 
\begin{equation}
\begin{split}
\bb E_{\muss}^n [ | I_T^n(V; f) | ]
	\leq \frac{2 d \rast \|\vec g\|_\infty}{\lambda_c n^{d/2}} \Big( H(\muss| \nurnast) + \log 2 + \max_{\pm} \log \bb E^n_{\nurnast} \big[ \exp \big\{ \pm \theta  I_T^n(V; f)   \big\}\big] \Big).
\end{split}
\end{equation}
To estimate the rightmost term in last display we use the following result. 
\begin{proposition}
\label{Feyn}
For every $T >0$ and every $W: [0,T] \times \Omega_n \to \bb R$,
\[
\log \bb E^n_{\nurnast} \Big[ \exp\Big\{ \int_0^T W_t(\eta(t)) dt \Big\} \Big]
	\leq \int_0^T \sup_{f} \Big\{ \int \Big(W_t + \tfrac{1}{2} L_n^\ast \mathbf{1} \Big) f d \nurnast - \int \Gamma_n \sqrt{f} d \nurnast\Big\}dt,
\]
where the supremum is taken over all densities $f$ with respect to $\nurnast$.
\end{proposition}

This proposition corresponds to Lemma A.1.7.2 of \cite{KipLan} in the case in which the measure $\nurnast$ is reversible with respect to $L_n$ (and therefore $L_n^\ast \mathbf{1} =0$). It was observed in \cite[p.78]{BenKipLan} that this estimate is also valid if the measure $\nurnast$ is invariant but not reversible.  A complete proof of Proposition \ref{Feyn} can be found in \cite[Lemma 3.3]{Sim} or  \cite[Lemma 9.1]{JarLan}.

The idea is to take $W_t = \theta n^{-d/2} V(\vec g_{t,T})$ in Proposition \ref{Feyn} and then to use Theorem  \ref{t4} to turn the variational formula into a quantitative estimate. Observe that for $W_t $, by \eqref{incahuasi} we get
\[
W_t + \frac{1}{2} L_n^\ast \mathbf{1} = V \big( \theta n^{-d/2} \vec g_{t,T} + \frac{\lambda}{2 d \rast} \vec{1}).
\]
Therefore, by Theorem \ref{main},
\[
\begin{split}
\int \!\!\Big( W_t + \frac{1}{2} L_n^\ast \mathbf{1} \Big) f d \nurnast\!\! -\!\! \int \Gamma_n \sqrt{f} d \nurnast
		 &\leq \Big( C \kappa(\rast) \mc A\Big( \frac{\theta }{n^{d/2}} \|\vec g_{t,T}\|_\infty + \frac{\lambda}{2d \rast}\Big) -1 \Big) \int \Gamma_n^r \sqrt{f}  d \nurnast\\
		 & \quad + C \mc A\Big( \frac{\theta }{n^{d/2}} \|\vec g_{t,T}\|_\infty + \frac{\lambda}{2 d \rast}\Big) g_d(n) n^{d-2},
\end{split} 
\]
for every $\lambda \in [-\lambda_c,\lambda_c]$, where the constant $\lambda_c$ was chosen in \eqref{laligua}. From our choice of $\theta$ 
we conclude that
\begin{equation*}
\int \Big( W_t + \frac{1}{2} L_n^\ast \mathbf{1} \Big) f d \nurnast - \int \Gamma_n \sqrt{f} d \nurnast \leq C \mc A \Big( \frac{\lambda_c}{d \rast}\Big) g_d(n) n^{d-2}.
\end{equation*} From the right-hand side of \eqref{laligua} we obtain that
\begin{equation}
\label{est4}
\begin{split}
\bb E_{\muss}^n [ | I_T^n(V; f) | ] \leq C \| f \|_\infty  g_d(n)n^{d/2 -2} (1+T).
\end{split}
\end{equation}

\subsection{Tightness}
\label{s4.2.2}

In this section we prove tightness of the sequence $(X^n; n \in \bb N)$ with respect to the strong topology of $\mc H^{-m}$ for $m$ large enough. By Prohorov's theorem, $(X^n; n \in \bb N)$ is tight in $\mc H^{-m}$ if for every $\epsilon >0$ there exists a compact set $K_\epsilon$ in $\mc H^{-m}$ such that
\[
\muss( X^n \notin K_\epsilon) \leq \epsilon
\]
for every $n \in \bb N$. It can be shown that whenever $m < m'$, the inclusion $\mc H^{-m} \subseteq \mc H^{-m'}$ is compact, which means that closed balls in $\mc \mc H^{-m}$ are compact in $\mc H^{-m'}$. Therefore, in order to prove tightness of $(X^n; n \in \bb N)$ in $\mc H^{-m'}$, it is enough to show that
\[
\lim_{M \to \infty} \sup_{n \in \bb N} \muss(\|X^n\|_{\mc H^{-m}} > M) =0
\]
for some $m < m'$. Observe that for every $T >0$, the law of $X^n$ under $\muss$ is equal to the law of $X_T^n$ under $\bb P_\muss^n$. Therefore, it is enough to estimate the probabilities $\bb P_\muss^n( \|\cdot\|_{\mc H^{-m}} > M)$ for each of the terms on the right-hand side of \eqref{quilpue}.

Let $Y$ be a random variable with values in $\mc H^{-m}$. We have that
\[
\bb P_\muss^n(\|Y\|_{\mc H^{-m}} > M) \leq M^{-2} \bb E_\muss^n[ \|Y\|^2_{\mc H^{-m}}].
\]
Observe that
\[
|\hat Y(k)| \leq |Y(\cos(2 \pi k x))| + | Y(\sin(2 \pi kx))|. 
\]
Therefore,
\[
\bb E_\muss^n[ \|Y\|_{\mc H^{-m}}^2] \leq \sum_{k \in \bb Z^d} (1+\|k\|^2)^{-m}  \big(Y(\cos(2 \pi k x))^2 +  Y(\sin(2 \pi kx))^2 \big).
\]
The relation
\[
f \mapsto X_0^n(P_T f)
\]
defines by duality a random variable in $\mc H^{-m}$, which for simplicity we just denote by $X_0^n(g_{0,T})$.
Using \eqref{est1}, we see that 
\[
\begin{split}
\bb E_\muss^n[\|X_0^n(g_{0,T})\|_{\mc H^{-m}}^2] 
	&\leq C g_d(n) n^{d-2} e^{2 F'(\rast) T} \sum_{k \in \bb Z^d} (1+\|k\|^2)^{-m}\\
	&\leq C g_d(n) n^{d-2} e^{2 F'(\rast) T}
\end{split}
\]
for every $m > d/2$. Recall that our aim is to provide a uniform bound for $\bb E_\muss^n[\|X_0^n(g_{0,T})\|_{\mc H^{-m}}^2]$. This is the case if we choose
\begin{equation}
\label{tn}
T = T_n := \frac{\log ( g_d(n) n^{d-2})}{-2 F'(\rast)}.
\end{equation}
Now let us estimate the second moment in $\mc H^{-m}$ of the process 
\begin{equation}
\label{llayllay}
f \mapsto \int_0^T X_t^n((\partial_t+\bb L_n) g_{t,T}) dt.
\end{equation}
Observe that for every $k, k' \in \bb Z^d$,
\[
\sin(2\pi (k+nk')x) = \sin(2 \pi kx), \quad \cos(2\pi (k+nk')x) = \cos(2 \pi kx).
\]
Therefore, $\hat X^n(k+nk') = \hat X^n(k)$.  The same identity holds for all the processes appearing in \eqref{quilpue} and in particular for the process defined in \eqref{llayllay}.
Take $k \in \bb Z^d$ such that $\|k\| \leq n$. For $f(x) = \sin(2 \pi k x)$ or $f = \cos(2 \pi k x)$, \eqref{est2} gives the bound
\[
\bb E_{\muss}^n \Big[ \Big|  \int_0^T X_t^n((\partial_t+\bb L_n) g_{t,T}) dt \Big|^2\Big] \leq C g_d(n) n^{d-6} \|k\|^8.
\]
For $k \in \bb N$ with $\|k\| >n$ there exists $k' \in \bb Z^d$ such that $\|k -n k' \|\leq n$. Therefore, \eqref{est2} gives us the bound
\[
\bb E_{\muss}^n \Big[ \Big|  \int_0^T X_t^n((\partial_t+\bb L_n) g_{t,T}) dt \Big|^2\Big] \leq C g_d(n) n^{d-6} n^8 \leq C g_d(n) n^{d+2}.
\]
We conclude that 
\begin{equation}
\label{rancagua}
\begin{split}
\bb E_{\muss}^n \Big[ \Big\|  \int_0^T X_t^n((\partial_t+\bb L_n) g_{t,T}) dt \Big\|^2_{\mc H^{-m}}\Big] 
		&\leq C g_d(n)n^{d-6} \Big( \sum_{\|k\| \leq n} \frac{\|k\|^8}{(1+\|k\|^2)^{m}}\\
		&\quad + \sum_{\|k\| > n} \frac{n^8}{(1+\|k\|^2)^{m}}\Big)\\
		&\leq C g_d(n) n^{2d+2 -2m},
\end{split}
\end{equation}
whenever $m > d/2$, so that the sum $\sum_{\|k|>n} (1+\|k\|^2)^{m}$ is finite. In particular, \eqref{rancagua} is bounded in $n$ for $m > 3$ if $d=2$ and for $m \geq 4$ if $d=3$. Observe that the bound is independent of $T$, and in particular it holds for $T=T_n$ as defined in \eqref{tn}.

Now we will prove tightness of $\mc M_{T_n,T_n}^n(f)$ for $T_n$ chosen as in \eqref{tn}. By \eqref{est3}, we have that
\[
\bb E_{\muss}^n\big[ \big\|\mc M_{T,T}^n \big\|_{\mc H^{-m}}^2\big] \leq C \sum_{k \in \bb Z^d} (1+\|k\|^2)^{-m}k^2 \leq C
\]
if $m > 1+ d/2$. Since the constant $C$ does not depend on $T$, tightness is proved. 
Observe that the cut-off introduced in \eqref{rancagua} does not improve the value of $m$.

We are only left to show tightness of the integral term $I_T^n(V)$.
Since we only have $L^1(\bb P_{\muss}^n)$ bounds for $I_T^n(V; f)$, we need to estimate $\bb P_\muss^n(\|I_T^n(V)\|_{\mc H^{-m}} > M)$ in a different way. Let $(p(k); k \in \bb Z^d)$ be a probability measure to be chosen in a few lines. We have that
\[
\begin{split}
\bb P_\muss^n(\|I_T^n(V)\|_{\mc H^{-m}} \geq M ) 
		&= \bb P_\muss^n(\|I_T^n(V)\|_{\mc H^{-m}}^2 \geq M^2 ) \\
		&= \bb P_\muss^n \Big( \sum_{k \in \bb Z^d} |\widehat{I_T^n(V)} (k)|^2 (1+ \|k\|^2)^{-m} \geq \sum_{k \in \bb Z^d} p(k) M^2\Big) \\
		&\leq \sum_{k \in \bb Z^d} \bb P_\muss^n \big( |\widehat{I_T^n(V)}(k)|^2 (1+\|k\|^2)^{-m} \geq p(k) M^2 \big)\\
		&\leq \sum_{k \in \bb Z^d} \frac{(1+\|k\|^2)^{-m/2} \bb E_{\muss}^n [|\widehat{I_T^n(V)}(k)|]}{p(k)^{1/2}M}.
\end{split}
\]
Using \eqref{est4} we obtain the bound
\[
\bb P_\muss^n \big(\|I_T^n(V)\|_{\mc H^{-m}} \geq M \big) \leq C g_d(n) n^{d/2-2}(1+T)\sum_{k \in \bb Z^d} \frac{(1+\|k\|^2)^{-m/2}}{M p(k)^{1/2}}.
\]
Choose $p(k) = c(1+\|k\|^2)^{-d/2+\delta}$ for $\delta >0$, where $c$ is the corresponding normalization constant. We obtain the bound
\[
\begin{split}
\bb P_\muss^n(\|I_T^n(V)\|_{\mc H^{-m}} \geq M ) 
	&\leq C g_d(n) n^{d/2-2}(1+T) \sum_{k \in \bb T^d} (1+\|k\|^2)^{-m/2+d/2+\delta} \\
	&\leq  C g_d(n) n^{d/2-2}(1+T)
\end{split}
\]
for every $m> 3d/2$. We conclude that
\[
\lim_{n \to \infty} \bb P_\muss^n(\|I_{T_n}(W)\|_{\mc H^{-m}} \geq M ) \leq \lim_{n \to \infty} C g_d(n) n^{d/2-2}\Big(1+ \frac{\log ( g_d(n) n^{d-2})}{-2 F'(\rast)}\Big)=0,
\]
from where $I_{T_n}(W)$ converges to $0$ in probability with respect to $L^1(\bb P_{\muss}^n)$. Since convergence in probability implies tightness, $I_{T_n}(W)$ is tight in $\mc H^{-m}$ for $m >3d/2$. We conclude that $(X^n; n \in \bb N)$ is tight in $\mc H^{-m}$ for $m > 3$ in $d=2$ and $m > 9/2$ in $d =3$.
\subsection{Convergence}
\label{s4.2.3}
In this section we will complete the proof of Theorem \ref{t3} by proving the convergence of the finite-dimensional distributions of $(X^n; n \in \bb N)$ to $X_\infty$.

Putting estimates \eqref{est1}, \eqref{est2} and \eqref{est4} into \eqref{quilpue}, we see that
\[
\begin{split}
\bb E^n_{\muss} \big[ \big| X_{T_n}^n(f) - \mc M_{T_n,T_n}^n(f) \big| \big] 
		&\leq Cg_d(n)^{1/2} n^{d/2-1} \|f\|_\infty e^{2 F'(\rast) T_n} \\
		&\quad +C g_d(n)^{1/2} n^{d/2-3} \big( \|D^2 f\|_\infty + \| D^4 f\|_\infty \big)\\
		&\quad \quad + C \|f\|_\infty g_d(n) n^{d/2-2}(1+T_n).
\end{split}
\]
For the choice of $T_n$ made in \eqref{tn}, we see that
\[
\lim_{n \to \infty} \big( X_{T_n}^n(f) - \mc M_{T_n,T_n}^n(f) \big) =0.
\]
Therefore, we have reduced the problem of convergence of $(X^n; n \in \bb N)$ to the problem of convergence of $( \mc M_{T_n,T_n}^n; n \in \bb N)$. We have already proved that the sequence $(\mc M_{T_n,T_n}^n; n \in \bb N)$ is tight. Therefore, we only need to obtain the limits of the sequences $(\mc M_{T_n,T_n}^n(f); n \in \bb N)$ for every test function $f$. We will use the following martingale convergence theorem:
\begin{proposition}
\label{p3} 
Let $(M_t^n; t \in [0,T])_{n \in \bb N}$ be a sequence of $\mc L^2$-martingales. Let
\[
\Delta_t^n := \sup_{0 \leq t \leq T} | M_{t}^n-M_{t-}^n| 
\]
be the size of the largest jump of $(M_t^n; t \in [0,T])$ on the interval $[0,T]$ and let $\Phi:[0,T] \to [0,\infty)$ be a continuous increasing function.
Assume that
\begin{itemize}
\item
[i)]
as $n \to \infty$, $\displaystyle{| \Delta_t^n| \to 0}$ in probability,
\item
[ii)] as $n \to \infty$, $\displaystyle{\<M^n\>_t \to \Phi(t)}$ in probability for every $t \in [0,T]$.
\end{itemize}
Under these conditions,
\[
M_t^n \to M_t
\]
in law with respect to the $J_1$-Skorohod topology of $\mc D([0,T], \bb R)$, where $(M_t; t \geq 0)$ is a continuous martingale of quadratic variation $\<M\>_t = \Phi(t)$ for every $t \in [0,T]$.
\end{proposition}

This  proposition is a restatement of Theorem VIII.3.11 of \cite{JacShi}. Observe that in this proposition, the time window $[0,T]$ is fixed, while in our situation $T_n$ grows with $n$. We will overcome this difficulty in the following way. Let $T >0$ be fixed and assume without loss of generality that $T_n \geq T$. Define $(M_t^n; t \in [0,T])$ as
\[
M_t^n := \mc M_{T_n-T+t,T_n}^n(f)
\]
for every $t \in [0,T]$. Observe that $(M_t^n; t \in [0,T])$ is a martingale. Its quadratic variation is equal to 
\[
\<M^n\>_t = \int_0^t \Gamma_n X^n_{T_n-T+s}(g_{T_n-T+s,T_n}) ds.
\]
Recall \eqref{quillota} and recall the definition of $G(\rho)$ given in \eqref{atacama}. For every $g: \bb T_n^d \to \bb R$, let us define
\[
\varphi_n(g) := \frac{2\chi(\rast)}{n^d} \sumxy n^2 \big( g(y)-g(x)\big)^2 + \frac{G(\rast)}{n^d} \sumx g(x)^2.
\]
By Theorem \ref{t4}, we have that
\[
\bb E^n_{\muss} \big[ \big( \Gamma_n X_{T_n\!-T+s}^n(g_{T_n\!-T+s,T_n\!}) - \varphi_n(g_{T_n\!-T+s,T_n\!}) \big)^2 \big] 
	\!\leq\! \frac{C g_d(n)( \|g_{T_n\!-T+s,T_n\!}\|_\infty^2\! +\! \|\nabla g_{T_n\!-T+s,T_n\!}\|_\infty^2)}{n^2}.
\]
For simplicity we have estimated the $\ell_n^2$ bounds of the functions appearing in \eqref{quillota} by their $\ell^\infty$ bounds; we are only missing a multiplicative constant that does not affect the computations. Observe that
\[
\lim_{n \to \infty} \varphi_n(g_{T_n-T+s,T_n}) =: \widetilde{\varphi}_{T-s}(f) = 2 \chi(\rast)\! \int \!\| \nabla P_{T-s} f(x)\|^2 dx + G(\rast) \!\int \! P_{T-s} f(x)^2 dx.
\]
By Taylor's formula and the trapezoidal rule,
\[
\begin{split}
\Big| \frac{1}{n^d} \sumxy n^2 \big(g_{T_n-T+s,T_n}\big(\tfrac{y}{n}\big) -
		& g_{T_n-T+s,T_n}\big(\tfrac{x}{n}\big)\big)^2 - \int \|\nabla P_{T-s}f(x)\|^2 dx\Big| \leq\\
		&\leq \frac{C}{n^2} \big( \|D P_{T-s}f\|_\infty \|D^3 P_{T-s}f\|_\infty + \|D^2 P_{T-s}f\|_\infty^2\big)\\
		&\leq \frac{C}{n^2} \big( \|D f\|_\infty \|D^3 f\|_\infty + \|D f\|_\infty^2 \big)
\end{split}
\]
and
\[
\begin{split}
\Big|\frac{1}{n^d} \sumxy g_{T_n-T+s,T_n}\big(\tfrac{x}{n}\big)^2 - \int P_{T-s}f(x)^2 dx \Big| 
		&\leq \frac{C}{n^2} \big( \|P_{T-s}f\|_\infty \|D^2P_{T-s}f\|_\infty + \| D P_{T-s}f \|_\infty^2 \big)\\
		&\leq \frac{C}{n^2} \big( \|f\|_\infty \|D^2 f\|_\infty + \| D f \|_\infty^2 \big).\\ 
\end{split}
\]
Since these estimates are uniform in $s$, we conclude that
\[
\lim_{n \to \infty} \< M^n\>_t = \int_0^t \widetilde{\varphi}_{T-s}(f) ds
\]
in $\mc L^2(\bb P_{\muss}^n)$, and condition ii) of Proposition \ref{p3} is satisfied for 
\[
\Phi(t) = \Phi_t(f) := \int_0^t \widetilde{\varphi}_{T-s}(f) ds.
\]
The jumps of $(M_t^n; t \in [0,T])_{n\in\mathbb N}$ at time $t \in [0,T]$ have size at most
\[
n^{-d/2} \|g_{T_n-T+t,T_n}\|_\infty \leq n^{-d/2} \|f\|_\infty,
\]
and in particular condition i) of Proposition \ref{p3} is satisfied. We conclude that $(M_t^n; t \geq 0)_{n\in\mathbb N}$ converges in law to a continuous martingale of quadratic variation $\{\Phi(t); t \in [0,T]\}$. In particular, the limit in law
\[
\lim_{n \to \infty} \big( \mc M_{T_n,T_n}^n(f) - \mc M_{T_n-T,T_n}^n (f)\big) = \lim_{n \to \infty} \big(M_T^n - M_0^n \big)
\]
exists and it is equal to $\mc N(0,\Phi_T(f))$.
Now let us estimate the variance of $\mc M_{T_n-T,T_n}^n(f)$. We have that
\[
\begin{split}
\bb E^n_{\muss} \big[ \mc M_{T_n-T,T_n}(f)^2\big] 
		&= \bb E^n_{\muss} \big[ \< \mc M_{T_n-T,T_n}^n(f)\> \big]
		= \int_0^{T_n-T} \bb E^n_\muss \big[ \Gamma_n X_s^n(g_{s,T_n})\big] ds\\
		&\leq \int_0^{T_n-T} \!\!\!\!\! C \big( \|g_{s,T_n}\|_\infty^2 + \|\nabla g_{s,T_n} \|_\infty^2 \big) ds\\
		&\leq \int_0^{T_n-T} C \big( \|f\|_\infty^2 + \|\nabla f\|_\infty^2\big) e^{2 F'(\rast) (T_n-s)} ds\\
		&\leq  C \big( \|f\|_\infty^2 + \|\nabla f\|_\infty^2\big) e^{2 F'(\rast) T} .
\end{split}
\]

Since $F'(\rast) <0$, the last line of this estimate converges to $0$ as $T \to \infty$, uniformly in $n$. Therefore, we conclude that
\[
\lim_{n \to \infty} \mc M_{T_n,T_n}^n(f) = \mc N(0, \Phi_\infty(f)),
\]
where
\[
\Phi_\infty(f) := 2 \chi(\rast) \int_0^\infty \int \| \nabla P_t f(x)\|^2 dx dt + G(\rast) \int_0^\infty \int P_tf(x)^2 dx dt.
\]
Observe that $\Phi_\infty(f)$ is also given by \eqref{sanfelipe}.
Now we are ready to complete the proof of Theorem \ref{t3}.
By Wald's device, the law of $X_\infty$ is characterized by the laws of $X_\infty(f)$ for $f \in \mc C^\infty(\bb T^d)$. Let $\widetilde{X}$ be a limit point of $(X^n; n \in \bb N)$, which exists by tightness, and let $f \in \mc C^\infty(\bb T^d)$. We have just proved that $\widetilde{X}(f)$ has the same law of $X_\infty(f)$. We conclude that $\widetilde{X} = X_\infty$, and in particular $(X^n; n \in \bb N)$ has a unique limit point, which finishes the proof of Theorem \ref{t3}.

\section{Discussion}

We have proved that the CLT fluctuations of the density of particles under the NESS of an example of a driven-diffusive model are given by a Gaussian process presenting non-local correlations. For the sake of clarity, we have presented the proof for one of the simplest driven-diffusive models for which a non-trivial limit appears, that is, for a reaction-diffusion model with quadratic interactions. The restriction to quadratic interactions is not essential; however the restrictions to small parameters $\lambda$ and dimensions $d \leq 3$ could not be removed without new arguments. 

\subsection*{General reaction rates}

For $\lambda \geq 0$, the reaction rates $c_x$ can be written as $c_x = h_x + \lambda r_x$, where 
\[
h_x(\eta) := a (1-\eta_x) + b \eta_x
\]
can be interpreted as a chemical reservoir of particles at density $\frac{a}{a+b}$, and
\[
r_x(\eta) = \frac{1}{2d} \sum_{y \sim x} \eta_y(1-\eta_x)
\]
can be interpreted as the interaction rate of some chemical reaction. The parameter $\lambda$ corresponds to the intensity of the reaction, and our smallness condition in $\lambda$ corresponds to assuming that the strength of the chemical reaction is uniformly bounded in $n$, with respect to the strength of the chemical bath, by some constant depending only on the density of the bath. For $\lambda \in (-a,0)$ some adjustments are needed, but a similar interpretation is possible. Our results hold for general local reaction rates $r_x$, but always under the boundedness assumption described above.

\subsection*{MFT and density fluctuations} The correlation operator of the limiting fluctuation density field can be described in terms of thermodynamic functions appearing in MFT. The MFT for reaction-diffusion models like the one presented here has been described in great generality in \cite{LanTsu}. One starts from the hydrodynamic equation
\[
\partial_t \rho = \nabla \cdot \big( D(\rho) \nabla \rho \big) + A(\rho) - B(\rho),
\]
where $D(\rho)$ is the diffusivity of the diffusion dynamics, equal to $1$ in our case,  $A(\rho) := \int c_x(\eta) (1-\eta_x) d \nu_\rho^n$ and $B(\rho) := \int c_x(\eta) \eta_x d \nu_\rho^n$. Then, under the invariant measure $\muss$, the density of particles is concentrated on stable solutions of the elliptic equation 
\begin{equation}
\label{elli}
\nabla \cdot \big( D(\rho) \nabla \rho \big) + A(\rho) - B(\rho) =0.
\end{equation}
Observe that in our notation, $F = A-B$ and $G =A+B$. If $F$ has a unique zero $\rast$ on the interval $[0,1]$, then the elliptic equation \eqref{elli} has a unique solution, which is the constant solution equal to $\rast$. In \cite{LanTsu}, large deviations principles were proved for the hydrodynamic and hydrostatic limit of the density of particles of the reaction-diffusion model under the assumption $\lambda \geq 0$. At a formal level, the correlation operator of the CLT for the density of particles corresponds to the Hessian of the rate function at its unique zero. In our infinite-dimensional context, such expansion is difficult to justify, but \emph{a posteriori} it provides the right answer. In \cite{JarMen} it has been shown that the fluctuations around the hydrodynamic limit follow the stochastic heat equation
\begin{equation}
\tag{SHE}
\label{SHE}
\begin{split}
\partial_t X 
		&= \nabla \cdot \big( D(\rast) \nabla X + \sqrt{\chi(\rast)}  \dot{\mc W}^1\big) + F'(\rast) X + \sqrt{G(\rast)} \dot{\mc W}^2,
\end{split}
\end{equation}
where $\dot{\mc W}^1$ is a space-time white noise with values in $\bb R^d$, $\dot{\mc W}^2$ is a space-time white noise independent of $\dot{\mc W}^1$ and $\chi(\rast)  = \rast(1-\rast)$ is the \emph{mobility} of the model. The density fluctuations around the hydrostatic limit are given by the stationary solution of this equation. The stationary solution of \eqref{SHE} can be obtained from \emph{Duhamel's formula}:
\[
X_\infty = \int_0^\infty \sqrt{2\chi(\rast)} \dot{\mc W}^1 \cdot \nabla P_t + \int_0^\infty \sqrt{G(\rast)} \dot{\mc W}^2 P_t,
\]
where $P_t := e^{(D(\rast) \Delta + F'(\rast))t}$. Therefore, $X_\infty$ is a centered, Gaussian process satisfying
\[
\bb E[X_\infty(f)^2] = \int_0^\infty \big( \chi(\rast) \|\nabla P_t f\|^2 + G(\rast) \| P_t f\|^2\big) dt.
\]
Observe that 
\[
\begin{split}
\int_0^\infty \| \nabla P_t f\|^2 dt 
		&= \int_0^\infty \< P_t f , - \Delta P_t f\> dt = \int_0^\infty \big( -\tfrac{1}{2} \partial_t \<P_t f ,P_t f\> + F'(\rast) \<P_t f, P_t f\>\big) dt \\
		&= \tfrac{1}{2} \|f\|^2 + F'(\rast) \int_0^\infty \|P_t f\|^2 dt.
\end{split}
\]
In particular, 
\[
\bb E[X_\infty(f)^2] = \chi(\rast)\|f\|^2 + (G(\rast) + 2 \chi(\rast) F'(\rast))\int_0^\infty \|P_t f\|^2 dt.
\]
One can check that $\lambda$ and $G(\rast) + 2 \chi(\rast) F'(\rast)$ have the same sign. Therefore, the sign of $\lambda$ indicates whether $X_\infty$ has positive or negative correlations. We observe that for $\lambda >0$, one can interpret $X_\infty$ as the sum of a white noise and an independent \emph{massive Gaussian free field}. For $\lambda <0$, it turns out that the sum of $X_\infty$ and an independent massive free field is equal to a white noise. 

\subsection*{Local equilibrium}

Since diffusion has a diffusive scaling and reaction does not have scaling, at small scales the effect of the diffusive dynamics is stronger than the effect of the reaction dynamics. Therefore, it is reasonable to expect the reaction-diffusion model to be in \emph{local equilibrium}, meaning that at small scales the law of the density of particles should look like the invariant law of the diffusive dynamics. In our case, this corresponds to a product Bernoulli measure with constant density. In Theorem \ref{loceq}, we have proved that this is indeed the case for the NESS, and we have established a mesoscopic scale up to which local equilibrium holds. In dimension $d=1$, we have shown that non-equilibrium behavior only appears at macroscopic scales, since at every mesoscopic scale, the NESS is statistically indistinguishable from a product Bernoulli measure with density $\rast$. We conjecture that Theorem \ref{loceq} is \emph{sharp}, that is, the scale $R_n$ is the largest scale at which one should expect the NESS to behave as a local equilibrium. Observe that for $d \geq 3$ the covariance operator of the massive free field has a singularity of order $\|y-x\|^{2-d}$ at the diagonal ($\log \|y-x\|$ for $d=2$), which is compatible with the mesoscopic scale $R_n$ stated in Theorem \ref{loceq}. 

Local equilibrium is one of the axioms from which MFT is derived \cite{BerD-SGabJ-LLan2}. Although the validity of MFT has been proved for a range of models, we are not aware of any proof of Theorem \ref{loceq}. Therefore, Theorem \ref{loceq} can be seen as the mathematical basis of the derivation of MFT from first principles.

\subsection*{Absolute continuity of fluctuations with respect to white noise}

In this section we prove that the limit fluctuations given by $X_\infty$ are absolutely continuous with respect to a white noise. Recall that our results are for $d\leq 3$. 
 Let $(\zeta_{k,i}; k \in \bb Z^d, i =1,2)$ be a sequence of i.i.d.~random variables with common law $\mc N(0,1)$ and define $(\xi_k; k \in \bb Z^d)$ as follows. For $k =0$,  $\xi_0=\zeta_{0,1}$, and for $k \in \bb Z^d$ with $k_1 >0$, 
\[
\xi_k := \frac{\zeta_{k,1} + i \zeta_{k,2}}{\sqrt 2} \text{ and }
\xi_{-k} := \frac{\zeta_{k,1} - i \zeta_{k,2}}{\sqrt 2}.
\]
From \eqref{sanfelipe}, $X_\infty$ has the representation
\[
X_\infty(f) = \sum_{k \in \bb Z^d}  \hat{f}(k) \xi_{k}  \sqrt{\lambda_k},
\]
where 
\[
\lambda_k := \frac{4 \pi^2 k^2\chi(\rast)}{4 \pi^2 k^2-F'(\rast)} + \frac{G(\rast)}{8 \pi^2 k^2 -2 F'(\rast)}.
\]
The white noise of variance $\chi(\rast)$ has the representation
\[
Y(f) = \sum_{k \in \bb Z^d}  \hat{f}(k) \xi_{k} \sqrt{\chi(\rast)}.
\]
In other words, in Fourier space, both processes $X_\infty$ and $Y$ are represented by tensor products of independent Gaussians. It will be convenient to rewrite the expressions for $X_\infty$ and $Y$ as
\[
X_\infty(f) := \hat{f}(0) \zeta_{0,1} \sqrt{\lambda_0} + \sum_{k_1 >0} \sqrt{2} \big( \Re ( \hat{f}(k)) \zeta_{k,1} + \Im (\hat{f}(k)) \zeta_{k,2} \big) \sqrt{\lambda_k},
\]
\[
Y(f) := \hat{f}(0) \zeta_{0,1} + \sum_{k_1 >0} \sqrt{2} \big( \Re ( \hat{f}(k)) \zeta_{k,1} + \Im (\hat{f}(k)) \zeta_{k,2} \big).
\]

The relative entropy between two centered Gaussians satisfies the formula
\[
H(\mc N(0,\sigma_2^2) | \mc N(0,\sigma_2^1)) = \Xi\Big( \frac{\sigma_2^2}{\sigma_1^2}\Big),
\]
where $\Xi(r) := \frac{1}{2} (r - \log r -1)$. In particular, the relative entropy is of order $\frac{1}{4} (\frac{\sigma_2^2}{\sigma_1^2}-1)^2$ as $\frac{\sigma_2^2}{\sigma_1^2} \to 1$, and therefore for every $M >0$ there exists a finite constant $C_M$ such that 
\[
H(\mc N(0,\sigma_2^2) | \mc N(0,\sigma_1^2)) \leq C_M \Big(\frac{\sigma_2^2}{\sigma_1^2}-1\Big)^2,
\]
whenever $ \frac{\sigma_2}{\sigma_1} \in [\frac{1}{M}, M]$. Recall \eqref{llullaillaco}. Therefore, there exists $\widetilde M$ finite such that $$\frac{\lambda_k}{\chi(\rast)} \in [\widetilde M^{-1}, \widetilde M]$$ for every $k \in \bb Z^d$. By the tensorization property of the relative entropy, the relative entropy of the law of $X_\infty$ with respect to the law of $Y$ is equal to 
\[
\sum_{k \in \bb Z^d} \Xi\Big(\frac{\lambda_k}{\chi(\rast)}\Big), 
\]
which turns to be bounded by
\[
\sum_{k \in \bb Z^d} C_{ \widetilde M}\Big( \frac{\lambda_k}{\chi(\rast)} -1 \Big)^2.
\]
Observe that there exists a constant $C$ such that
\[
\Big| \frac{\lambda_k}{\chi(\rast)} -1 \Big| \leq \frac{C}{1+ \|k\|^2}.
\]
Therefore, the sum 
\[
\sum_{k \in \bb Z^d} \Big( \frac{\lambda_k}{\chi(\rast)} -1 \Big)^2
\]
is finite for $d <4$, from where we conclude that $X_\infty$ has a density with respect to $Y$, as we wanted to show.

\appendix

\section{Lemmas}

The following estimate is known as \emph{Hoeffding's lemma}:

\begin{lemma}
 \label{Hoeff}
Let $X$ be a random variable such that $a\leq X \leq b$ a.s. For every $\theta
\in \bb R$,
\[
\log \bb E[e^{\theta (X - \bb E[X])}] \leq \tfrac{1}{8}(b-a)^2 \theta^2.
\]
\end{lemma}

The following lemma states that random variables satisfying the bound in
Hoeffding's lemma are subgaussian:

\begin{lemma}
 \label{subG}
Let $X$ be a real-valued random variable such that $\log \bb E[e^{\theta X}]
\leq \frac{1}{2}\sigma^2\theta^2$ for every $\theta \in \bb R$. We have that
\[
\bb E [e^{\gamma X^2}] \leq \frac{1}{\sqrt{1- 2\sigma^2 \gamma}}
\]
for every $\gamma < \frac{1}{2\sigma^2}$.%
\end{lemma}

\begin{proof}
 Let $Z$ be a Gaussian random variable of mean zero and unit variance,
independent of $X$. We have that
\[
\bb E[e^{\gamma X^2}] = \bb E[ e^{\sqrt{2 \gamma} X Z}] \leq \bb E[e^{\sigma^2
\theta Z^2}] = \frac{1}{\sqrt{1-2\sigma^2 \gamma}}.
\]
\end{proof}

\thanks{ {\bf{Acknowledgements:}}
P.G. thanks  FCT/Portugal for financial support
through CAMGSD, IST-ID, projects UIDB/04459/2020 and UIDP/04459/2020.  M.J.~has been funded by CNPq grant 312146/2021-3, CNPq grant 201384/2020-5 and FAPERJ grant E-26/201.031/2022. M.J.~was supported in part by funding from the Simons Foundation and the Centre de Recherches Math\' ematiques, through the Simons-CRM scr-in-residence program and by the Mathematisches Forschungsinstitut Oberwolfach. This project has received funding from the European Research Council (ERC) under  the European Union's Horizon 2020 research and innovative programme (grant agreement  n.~715734). R.M.~thanks the National Council for Scientific and Technological Development (CNPq).
Data sharing not applicable to this article as no datasets were generated or analysed during the current study.}

\bibliographystyle{plain}

\end{document}